\documentclass[a4paper]{amsart}
\usepackage[utf8]{inputenc}
\usepackage[T1]{fontenc}

\usepackage{mathpazo}
\usepackage{microtype}

\usepackage{amssymb}
\usepackage{amsfonts}
\usepackage{amsthm}
\usepackage{calrsfs}
\usepackage{array}
\usepackage{bbm}
\usepackage{stmaryrd}
\usepackage{hyperref}
\usepackage{slashed}

\usepackage{enumitem}

\usepackage{csquotes}

\usepackage{mathtools}

\newcommand\R{\mathbb{R}}
\newcommand\N{\mathbb{N}}

\newcommand\M{\mathbb{M}}

\newcommand\bneq{\begin{eqnarray*}\left\lbrace \begin{array}{rcl}}
\newcommand\eneq{\end{array} \right.\end{eqnarray*}}
\newcommand\bneqn{\begin{eqnarray}\left\lbrace \begin{array}{rcl}}
\newcommand\eneqn{\end{array} \right.\end{eqnarray}}

\newcommand\nor[2]{\left\|#1\right\|_{#2}}
\newcommand\sgn{\textnormal{sgn}}

\newcommand{\Ra}{\mathcal{R}}

\newcommand\Td{\mathcal{T}}
\newcommand\Nd{\mathcal{N}}

\newcommand{\m}[1]{
\ifdefequal{#1}{1}
{\mathbbm{#1}}
{\mathbb{#1}}
}

\newcommand{\q}[1]{\mathcal{#1}}

\newcommand{\ds}{\displaystyle}
\newcommand{\be}{\begin{gather}}
\newcommand{\ee}{\end{gather}}
\newcommand{\ba}{\begin{align*}}
\newcommand{\ea}{\end{align*}}

\newcommand{\e}{\varepsilon}

\newcommand{\w}{\omega}

\newcommand{\weak}{\rightharpoonup}

\DeclareMathOperator{\Span}{\mathrm{Span}}
\DeclareMathOperator{\Supp}{\mathrm{Supp}}

\renewcommand{\le}{\leqslant}
\renewcommand{\ge}{\geqslant}

\numberwithin{equation}{section}

\theoremstyle{plain}
\newtheorem{thm}{Theorem}[section]
\newtheorem*{thm*}{Theorem}

\newtheorem{prop}[thm]{Proposition}
\newtheorem{cor}[thm]{Corollary}
\newtheorem{lem}[thm]{Lemma}

\theoremstyle{definition}
\newtheorem{definition}[thm]{Definition}

\theoremstyle{remark}
\newtheorem{remark}[thm]{Remark}
\newtheorem{claim}[thm]{Claim}

\usepackage{xcolor}

\allowdisplaybreaks

\begin{document}
\parindent=0pt
\title{Concentration close to the cone for linear waves}
\author{Rapha\"el C\^{o}te}
\address{Institut de Recherche Math\'ematique Avanc\'ee, UMR 7501, Universit\'e de Strasbourg, 7 rue Ren\'e-Descartes, 67084 Strasbourg Cedex, France}
\email{cote@math.unistra.fr}
\author{Camille Laurent}
\address{CNRS UMR 7598 and Sorbonne Universit\'es UPMC Univ Paris 06, Laboratoire Jacques-Louis Lions, F-75005, Paris, France}
\email{camille.laurent@ljll.math.upmc.fr}
\begin{abstract}
%
We are concerned with solutions to the linear wave equation. We give an asymptotic formula for large time, valid in the energy space, via an operator related to the Radon transform. This allows us to show that the energy is concentrated near the light cone. This allows to derive further expressions the exterior energy (outside a shifted light cone). We in particular generalize the formulas of \cite{CoteKenigSchlagEquipart} obtained in the radial setting. In odd dimension, we study the discrepancy of the exterior energy regarding initial energy, and prove in the general case the results of \cite{KLLS15} (which were restricted to radial data).
\end{abstract}
\subjclass[2010]{35L05, 35B40} 
\keywords{linear wave equation, exterior cone, channel of energy, profile decomposition}
\maketitle
\tableofcontents

\medskip
\noindent
\section{Introduction and statement of the results}
\subsection{General results about asymptotic profile}

In this paper, we consider solutions to the linear wave equation in any dimension $d \ge 1$.
\begin{equation} \label{eq:lw}
\begin{cases}
\partial_{tt} u - \Delta u = 0, \\
(u,\partial_t u)_{\left|t=0\right.}=(u_0,u_1),
\end{cases} \quad
(t,x) \in \m R \times \m R^d.
\end{equation}

We are particularly interested in understanding how the energy of $w$ concentrate around the light cone for large times, that is, provide some formulas for quantities which are typically 
\[ \lim_{t \to +\infty} \| \nabla_{x,t} u \|_{\dot H^1 \times L^2(|x| \ge t+R} \]
where $R \in \m R$ is fixed, in terms of the initial data $(u_0,u_1)$. This kind of quantities are very natural when thinking of finite speed of propagation for solutions to the linear wave equation, but are also useful in nonlinear contexts, for example for the channels of energy method, we refer for example to \cite{DuyKMUniversnonrad} for one of the first time it was used in the context of the energy critical non linear wave equation. Such formula where given in the radial setting, notably in \cite{CoteKenigSchlagEquipart} and \cite{KLLS15}, and we aim at generalizing the result therein to non radial linear waves. 

\bigskip

We can formulate our first results on solution of the half-wave equation, that is, consider $e^{t |D|} f$, where $ |D|$ is the operator defined as a multiplier in Fourier space
\[ \widehat{|D| f}(\xi) = |\xi| \hat f(|\xi|), \]
where $\hat f$ is the $d$-dimensional Fourier transform of $f$:
\[ \hat{f}(\xi)=\int_{\R^d} e^{-ix\cdot \xi}f(x) dx. \]
For functions of several variables (say $s$ and other ones), we will consider in an analogous way $|D_s|$, where the Fourier transform is restricted to the $s$ variable.

Our results on the half-wave equation will transfer to the wave equation as its solutions can be written
\begin{gather} \label{def:w_fg}
u=e^{it|D|} f+e^{-it|D|}g \quad \text{where} \quad f :=\frac{1}{2}\left[ u_0+\frac{1}{i|D|}u_1\right] \text{ and }  g:=\frac{1}{2}\left[ u_0-\frac{1}{i|D|}u_1\right].
\end{gather}

We now introduce some notation. Given a function $f$ on $\m R^d$ and $\omega \in \m S^{d-1}$, let $f_\omega^{\pm} : \m R \to \m C$ be such that its 1-dimensional Fourier transform (as a function of $\nu \in \R$) is 
\begin{equation} \label{def:f_omega}
\q F_{\m R} (f_\omega^{\pm})(\nu) = \m 1_{\pm \nu \ge 0} |\nu|^{\frac{d-1}{2}}\hat{f}(\nu \omega).
\end{equation}
We also use the notation
\begin{gather} \label{def:tau}
\tau :=  \frac{d-1}{4} \pi \quad \text{and} \quad c_0 = \frac{1}{\sqrt{2(2\pi)^{d-1}}}. 
\end{gather}

Finally, we define the operator $\Td$ as follows: for a function $v$ defined on $\m R^d$, $\Td v$ is a function of two variables $(s,\omega)$, defined on $\m R \times \m S^{d-1}$ by its (partial) Fourier transform in the first variable $s$:
\begin{align}
\label{def:T}
\q F_{s \to \nu}  (\Td v)(\nu, \omega) = c_0  |\nu|^{\frac{d-1}{2}} (e^{i \tau}  \m1 _{\nu < 0}   + e^{-i\tau} \m 1_{\nu \ge 0} ) \hat v(\nu \omega).
\end{align}
that is,
\[ (\Td v)(s,\omega)=c_0 \left(e^{i \tau}v_\omega^{-}(s)+e^{-i \tau}v_\omega^{+}(s)\right). \]

Our first main result is the description of the asymptotic for large times of solutions of the half-wave equation, and then of the wave equation.

\begin{thm}[Radiation field and concentration of energy on the light cone] \label{th1}
{\ }

1) (Half-wave equation) Let $f \in L^2(\m R^d)$. Then as $t \to +\infty$, the convergence holds
\begin{gather}
\label{eq:expansion_w}
(e^{i t |D|} f)(x) -  \frac{e^{i\tau}}{(2\pi|x|)^{\frac{d-1}{2}}}  f_{x/|x|}^- (|x|-t) \to 0 \quad \text{in} \quad L^2(\m R^d). 
\end{gather}
Furthermore, one has
\begin{gather} \label{eq:conc_cone}
\limsup_{t \to \pm \infty} \| e^{i t |D|} f \|_{L^2(\left| |x| - |t| \right|\ge R)} \to 0 \quad \text{as} \quad R \to +\infty.
\end{gather}

2) (Wave equation) Let $(u_0, u_1) \in \dot H^1 \times L^2(\m R^d)$, and $u$ be the solution to \eqref{eq:lw}. Then as $t \to +\infty$, the convergence holds
\begin{gather} \label{eq:expansion_w_1}
\nabla_{t,x} u(t,x) - \frac{1}{\sqrt{2}|x|^{\frac{d-1}{2}}}(\partial_s \Td u_0 - \Td u_1) \left(|x|-t, \frac{x}{|x|} \right) \times \begin{pmatrix}
-1 \\
x/|x|
\end{pmatrix}  \to 0 \quad \text{in} \quad L^2(\m R^d)^{1+d}.
\end{gather}
Furthermore, one has
\begin{gather} \label{eq:conc_cone_2}
\limsup_{t \to \pm \infty} \| \nabla_{t,x} u(t) \|_{L^2(\left| |x| - |t| \right|\ge R)} \to 0 \quad \text{as} \quad R \to +\infty.
\end{gather}
\end{thm}

Of course, for $g \in L^2(\m R^d)$, one obtains the corresponding  expression for $e^{-i t|D|} g$ by considering the complex conjugate in \eqref{eq:expansion_w}:
\begin{gather}
\label{eq:expansion_g} 
(e^{-i t |D|} g)(x) -  \frac{e^{-i\tau}}{(2\pi|x|)^{\frac{d-1}{2}}}  g_{x/|x|}^+ (|x|-t) \underset{t\to +\infty}{\longrightarrow}  0 \quad \text{in} \quad L^2(\m R^d). 
\end{gather}
This also gives an expansion for $t \to -\infty$. Also, 2) is a rather direct consequence of 1), as we will prove the following equality which has its own interest:
\begin{align} \label{eq:fg_T}
(\partial_s \Td u_0 - \Td u_1)(s,\omega) = 2c_0 \partial_s(  e^{i \tau} f_{\omega}^- + e^{-i \tau} g_{\omega}^+)(s).
\end{align}

This result is therefore a computation of the radiation fields of Friedlander \cite{F:80}. We refer to \cite{BVW:18}; in odd dimension, it can be classically written thanks to the Radon transform (see \cite{MelroseBookGeom}, \cite{LaxPhillipsBook}). Actually, the operator $\Td$ is related to the Radon transform (which we recall below in \eqref{Radondef}), see section  \ref{sectRadsob} and Lemma \ref{lmRadon}. However, as far as we can tell, the correct computation of the convergence (in $L^2$) is new, specially in even dimension. Very recenlty, upon the completion of this work, Li, Shen and Wei performed a related analysis in \cite{LSW21}.


For example, as an easy consequence, we can also compute the energy outside a (shifted) light cone, or the asymptotic energy at $+\infty$ and $-\infty$ in the following sense.

\begin{definition}
Given $R \in \m R$ and a space time function $v$, we denote 
\begin{align*}
E_{ext,R}(v)& := \frac{1}{2} \left( \lim_{t \to +\infty} ( \nor{\nabla v}{L^2(|x| \ge t+R)}^2 +  \nor{\partial_t v}{L^2(|x| \ge t+R)}^2) \right.  \\
& \qquad \qquad \left. + \lim_{t \to -\infty} (\nor{\nabla v}{L^2(|x| \ge |t|+R)}^2 +  \nor{\partial_t v}{L^2(|x| \ge |t|+R)}^2) \right)
\end{align*}
assuming that the limits exist.
\end{definition}

Then there hold

\begin{cor}[Mass outside the light cone] \label{prop:ext_energy}
Let $R \in \m R$. We have the formula
\begin{align}
\MoveEqLeft \lim_{t \to +\infty} \nor{u}{L^2(|x| \ge t+R)}^2 \nonumber \\
&  = \frac{1}{(2\pi)^{d-1}}\int_{\omega \in \m S^{d-1}}\nor{f_\omega^-(s)+e^{-i\frac{\pi}{2}(d-1)} g_\omega^+(s)}{L^2([R,+\infty))}^2  d\omega. \label{equivfinallemma}
\end{align}
Also, in the case of an initial datum $(u_0,u_1) \in \dot H^1 \times L^2$ in the energy space, we have the formula
\begin{align} \label{equivfinallemmaH1}
\MoveEqLeft \lim_{t \to +\infty} \nor{\nabla u}{L^2(|x| \ge t + R)}^2 = \lim_{t \to +\infty} \nor{\partial_{t} u}{L^2(|x| \ge t + R)}^2  \\
& =\frac{1}{(2\pi)^{d-1}}\int_{\omega \in \m S^{d-1}}\nor{e^{i\tau} \partial_{s} f_\omega^-(s)+e^{-i\tau} \partial_{s}g_\omega^+(s)}{L^2([R,+\infty))}^2  d\omega. \nonumber \\
& = \frac{1}{2} \|  \partial_{s}\Td u_{0} - \Td u_{1} \|_{L^2([R,+\infty) \times \m S^{d-1})}^2.  \label{equivfinallemmaH1_2} 
\end{align}
(The first equality in \eqref{equivfinallemmaH1} is equipartition). As a consequence, there is asymptotic orthogonality in the sense that
\begin{align}
\label{eq:ortho_u0_u1}
E_{ext,R}(u) =  \|  \partial_{s}\Td u_{0} \|_{L^2([R,+\infty) \times \m S^{d-1})}^2 + \| \Td u_{1} \|_{L^2([R,+\infty) \times \m S^{d-1})}^2.
\end{align}
\end{cor}

(Here and below, $\m R \times \m S^{d-1}$ is equipped with the standard product measure).
The last two formulas \eqref{equivfinallemmaH1_2} and \eqref{eq:ortho_u0_u1} involving $\Td$ reveal the important role of this operator in our analysis. We can reformulate \eqref{eq:ortho_u0_u1} in the following way: denoting $u^e$ [resp. $u^o$] the solution to \eqref{eq:lw} with initial data $(u_0,0)$ [resp. $(0,u_1)$] then
\[ E_{ext,R}(u) = E_{ext,R}(u^e) + E_{ext,R}(u^o). \]

Another consequence of Theorem \ref{th1} is related to profile decomposition in the sense of Bahouri-G\'erard \cite{BahouriGerard}, for which we can prove Pythagorean expansion of the linear energy \emph{with sharp cut-off}: this in turn is useful for the channel of energy method. Let us recall the notion of profile decomposition in our setting.

\begin{definition}
Let $(u_{n,0},u_{n,1})$ be a bounded sequence in $\dot H^1 \times L^2$. We say that is admits a linear profile decomposition $(\vec U_L^j; (\lambda_{j,n})_n, (t_{j,n})_n,(x_{j,n})_n)_j$, with remainder $(\vec w_n^J)_{n,J}$ where the $\vec U_L^j$ and the $\vec w_{n}^J$ are solutions to the linear wave equation, and the parameters $(\lambda_{j,n})_n, (t_{j,n})_n,(x_{j,n})$ are sequences in $[0,+\infty)$, $\m R$ and $\m R^d$ respectively satisfies 
\begin{enumerate}[wide]
\item Decomposition: for all $J \ge 1$, there holds
\begin{align*}
\MoveEqLeft (u_{n,0},u_{n,1}) (x) \\
& = \sum_{j=1}^J \left( \frac{1}{\lambda^{d/2-1}} U_{L}^j\left( - \frac{t_{j,n}}{\lambda_{j,n}}, \frac{x-x_{j,n}}{\lambda_{j,n}} \right), \frac{1}{\lambda^{d/2}} \partial_t U_{L}^j\left( - \frac{t_{j,n}}{\lambda_{j,n}}, \frac{x-x_{j,n}}{\lambda_{j,n}} \right) \right) + \vec w_{n}^J(0),
\end{align*}
where the remainder converges in the adequate Strichartz space $S  = L^{\frac{2d}{d-2}}_{t,x}(\m R^{1+d})$ 
\[ \limsup_{n \to +\infty} \| \vec w_{n}^J \|_{S} \to 0 \quad \text{as} \quad  J \to +\infty. \]
\item Pseudo-orthogonality: for $j \ne k$, 
\begin{gather*}
\text{either:}  \qquad \frac{\lambda_{j,n}}{\lambda_{k,n}} +  \frac{\lambda_{k,n}}{\lambda_{j,n}} \to +\infty, \\
\text{or: } \quad  \forall n,  \ \lambda_{j,n} = \lambda_{k,n} \quad \text{and} \quad \frac{|t_{j,n} - t_{k,n}|}{\lambda_{j,n}} + \frac{|x_{j,n} - x_{k,n}|}{\lambda_{j,n}} \to +\infty.
\end{gather*}
\end{enumerate}
\end{definition}

\begin{prop}[Orthogonality with cut-offs in a profile decomposition] \label{prop:ortho_cutoff}
Let $(u_{n,0},u_{n,1})$ be a bounded sequence of $\dot H^1 \times L^2$, and assume that it admits a profile decomposition with waves and parameters $(\vec U_L^j; (\lambda_{j,n})_n, (t_{j,n})_n,(x_{j,n})_n)_j$, and remainder $(\vec w_n^J)_{n,J}$.

Let $(r_n)_n$ and $(x_n)$ be two sequences of $[0,+\infty)$ and $\m R^d$ respectively. Then 
\begin{align}
\MoveEqLeft \| (u_{n,0},u_{n,1})  \|_{\dot H^1 \times L^2(|x-x_n| \ge r_n)}^2  =  \sum_{j=1}^J  \left\| \nabla_{t,x} U_L^j \left( - \frac{t_{j,n}}{\lambda_{j,n}} \right) \right\|^2_{L^2(|\lambda_{j,n} x + x_{j,n} - x_n| \ge r_n)} \nonumber \\
& +  \| (w_{n,0}^J,w_{n,1}^J)  \|^2_{\dot H^1 \times L^2(|x-x_n| \ge r_n)} + o_n(1). \label{eq:ortho_cutoff}
\end{align}
\end{prop}

Such a decomposition with sharp cut-off is relevant for the channel of energy method in a non linear setting (see \cite{DuyKMUniversnonrad} for example), and was done in the radial setting in \cite{CoteKenigSchlagEquipart}.

Note that an interesting byproduct of the proof is an explicit formula for reconstructing the initial datum of a solution of the wave equation from its radiation field described in \eqref{eq:expansion_w_1}, see \eqref{reconstructfg}, \eqref{reconstructv01}.

\subsection{Odd dimension}

In odd dimension, we are able to precise the previous results and the asymptotic energy outside truncated cones $|x| \ge t+R$ with $R\ge 0$. 

We first consider the easier case $R=0$. From our computations, we can easily recover the following result, which goes back at least to Duyckaerts, Kenig and Merle \cite{DuyKMUniversnonrad}.

\begin{prop}
\label{thmexteriorgeneR0}
Assume $d$ odd and $u$ be a a solution to \eqref{eq:lw} with initial data $(u_0,u_1)\in\dot H^1 \times L^2(\m R^d)$. Then, we have
\begin{align}
E_{ext,0}(u) = \frac{1}{2}\nor{(u_0,u_1)}{\dot H^1 \times L^2(\R^{d})}^{2}. \label{channelgeneR0}
\end{align}
\end{prop}

Then, we consider the case $R>0$ where the previous result cannot hold. We are nonetheless able to determine the solutions $u$ that have vanishing asymptotic energy on the exterior light cone $|x| \ge t+R$ with $R> 0$, that is
\[ E_{ext,R}(u) =0. \]
By finite speed of propagation, initial data which are compactly supported in  $|x| \le R$ obviously satisfy this condition. We will call this space
\[
\mathcal{K}_{R,comp}= \left\{(u_0,u_1)\in\dot H^1 \times L^2 (\m R^d): (u_0,u_1) |_{\{ |x| > R \}} =0\right\}.
\]
where the equality is in the distributional sense.

It turns out that these are not the only examples. We will now need some further notation.

We denote 
\begin{align}
\label{def:Yl}
(Y_{\ell})_{\ell \in \M}
\end{align}
a countable orthonormal basis of spherical harmonics of $\m S^{d-1}$. $Y_{\ell}$ is the restriction to $\m S^{d-1}$ of a harmonic (homogeneous) polynomial. For short, we will denote $l=l(\ell)$ the degree of this polynomial.

The non radiative functions will be the following. Denote for $k \in \m N$,
\[ \alpha_{k}: = -l-d+2k+2. \]
$\alpha_k$ also depends on $\ell$, but here and below, we silence this dependence to keep notations light. Then let
\begin{align} \label{def:gk}
g_{k}(x)= \m 1_{\{|x|>R\}}|x|^{\alpha_{k}}Y_{\ell}\left(\frac{x}{|x|}\right)
\end{align}
Note that $g_k \in L^2 \iff \alpha_k <-d/2$. We introduce
\[
\Nd_{R, \ell}^{0}=\Span \left( g_{k} ;  \text{ for }  k \in \N \text{ such that }\alpha_{k}<-d/2\right)
\]
Similarly, let
\begin{align} \label{def:fk}
f_{k}(x)= \begin{cases}
\ds  \left( \frac{|x|}{R} \right)^{\alpha_{k}}Y_{\ell}\left(\frac{x}{|x|}\right) & \text{ for }|x|>R\\
\ds  \left( \frac{|x|}{R} \right)^{l}Y_{\ell}\left(\frac{x}{|x|}\right) & \text{ for }|x|\le R.
\end{cases}
\end{align}
Note that $f_k \in \dot H^{1} \iff \alpha_{k}<-d/2+1$. Also, the value of $f_{k}$ in $|x|\le R$ is not very important; our choice allows to keep continuity and that the restriction $f_{k}|_{\left\{|x| < R\right\}}$ is a harmonic polynomial, so that $f_{k}$ is orthogonal to (in $\dot H^1$) to functions with compact support in $B(0,R)$. Let
\[ \Nd_{R, \ell}^{1}=\Span \left( f_{k};  \text{ for } k \in \N \text{ such that } \alpha_{k}<-d/2+1\right). \]

For any $\ell\in \M$, we note the space
\[ 
P_{\ell}(R)=\Nd_{R, \ell}^{0} \times \Nd_{R, \ell}^{1}, \quad \text{and} \quad
P(R)= \mathcal{K}_{R,comp} \stackrel{\perp}{\oplus}  \bigoplus_{\ell\in \M}^\perp P_{\ell}(R), \]
(the orthogonality is related to the natural scalar product of $\dot H^1 \times L^2$).
Then we will prove that if $u$ is a linear wave solution which is non radiative, that is such that  $E_{ext,R}(u)=0$, then $(u,\partial_t u)|_{t=0} \in P(R)$ (and the converse is true as well). We actually have a quantitative version of this fact: this is our second main result.

\begin{thm}
\label{thmexteriorgene}
Assume $d$ is odd, $d\geq 3$, and let $R>0$. Let $u$ be the solution to the linear wave equation with initial data $(u,\partial_tu)_{\left|t=0\right.}=(u_0,u_1)\in\dot H^1 \times L^2 $. Then, we have
\begin{align}  
\label{channelprojectgene}
\nor{(u_0,u_1)}{\dot H^1 \times L^2(\R^{d})}^{2}= 2E_{ext,R}(u)+\nor{\pi_{R }(u_0,u_1)}{\dot H^1 \times L^2(\R^{d})}^{2}.
\end{align} 
where $\pi_{R}$ is the orthogonal projection (in $\dot H^1 \times L^2$) onto the space $P(R)$ .

Moreover,  if $(u_0,u_1)\in P(R)$, then the equality
\[ u (t,x) = \sum_{\ell\in \M}v_{\ell}(t,r)Y_{\ell}(\w) \]
holds for all $(t,x)$ in the (outer) truncated cone $\q C_{R}=\left\{(t,x)\in \R^d;|x|-|t|\ge R\right\}$, where
\begin{align}  
\label{formvsolgene}
v_{\ell}(t,r)=\sum_{j=1}^{B }\frac{1}{r^{d+l-j}}\sum_{i=0}^{B-j}d_{i,j}t^{i}.
\end{align} 
 for some $d_{i,j}\in \m C$, and where $B:=\frac{d+1}{2} + l$.
\end{thm}

The theorem above is the generalization to non radial data of the main result in \cite{KLLS15} (which was revisited by \cite{LSW21}).

\subsection{Even dimension}

In even dimension, we are able to give a more tractable formula for $E_{ext,0}(u)$. 

\begin{prop} \label{extEnergy_even}
Assume that $d$ is even and let $u$ be a solution to \eqref{eq:lw} with initial data $(u_0,u_1)\in\dot H^1 \times L^2(\m R^d)$. Then, we have
\begin{align*}
&E_{ext,0}(u) =\frac{1}{2}\nor{(u_0,u_1)}{\dot H^1 \times L^2(\R^{d})}^{2}\\
&+   \frac{(-1)^{\frac{d}{2}} }{(2\pi)^{d+1}}\Re\int_{\omega \in \m S^{d-1}} \int_0^\infty \int_0^\infty (rs)^{\frac{d-1}{2}} \frac{\overline{s \widehat{ u_{0}}(s \omega)}  r \widehat{ u_{0}}(-r \omega) -  \overline{ \widehat{ u_{1}}(s \omega)}  \widehat{ u_{1}}(-r \omega)}{r+s}drds d\omega.
\end{align*}
More precisely, there hold
\begin{align*}
\MoveEqLeft 2 \lim_{t \to +\infty} \nor{ \nabla u}{L^2(|x| \ge t)}^2 = E_{ext,0}(u) \\
&+ \frac{2}{(2\pi)^{d+1}} \Re\int_{\omega \in \m S^{d-1}}\int_0^\infty \int_0^\infty \frac{r^{\frac{d+1}{2}}\widehat{ u_{0}}(r \omega) s^{\frac{d-1}{2}}\overline{\widehat{ u_{1}}(s \omega)}}{r-s}drds d\omega.
\end{align*}
\end{prop}
This is therefore an equivalent of Proposition \ref{thmexteriorgeneR0}. It extends the results of \cite{CoteKenigSchlagEquipart} where this formula first appeared for radial data (to recover this formula, notice that $\hat u_{0}(-r) = \hat u_{0}(r)$ when $u_0$ is radial). Very recently, Delort also derives a similar formula in \cite{Delort21}.

\subsection{Outline and organisation of the paper}

The proof of Theorem \ref{th1} relies on an adequate stationary phase analysis, which is reminiscent of second microlocalisation. Our main input is a careful bound on the remainder term, to derive $L^2$ type convergence. Corollary  \ref{prop:ext_energy} and Propositions \ref{thmexteriorgeneR0} and \ref{extEnergy_even} are easy consequences; Proposition \ref{prop:ortho_cutoff} requires an extra ingredient, depending on the various cases that the cut-offs can take.

\bigskip

Our main goal is obviously Theorem \ref{thmexteriorgene}. The operator $\Td$ is related to the Radon transform $\q R$, which is defined as follows for a function $f\in \mathcal{S}(\R^{d})$
\begin{align}
\label{Radondef}
\forall (s,\omega) \in \m R \times \m S^{d-1}, \quad (\Ra f)(s,\omega):=\int_{\omega\cdot y=s}f(y)dy.
\end{align}
We will prove in section \ref{sectRadsob} that some variant of the operator $\Ra$ can be extended to a map $L^2(\m R^d) + \dot H^1(\m R^d) \to \q S'(\m R \times \m S^{d-1})$ and that, when $d$ is odd, one has the equality
\begin{align} \label{def:T_odd}
\Td = c_0 (-1)^{\frac{d-1}{2}}\partial_s^{\frac{d-1}{2}} \Ra,
\end{align}
seen as operators on $\q S(\m R^d)$.
We emphasize that $\partial_s^{\frac{d-1}{2}}$ is a differential operator, and so, in odd dimensions, $\Td$ enjoys similar locality properties featured by $\Ra$: this is a key aspect of the analysis. In order to retain these locality properties for data in $L^2$ (or $\dot H^1$ for $\partial_s \Td$), we cannot use Fourier analysis and instead proceed by duality. This is the purpose of section \ref{sectRadsob}.
A special attention is required by the fact that the Radon transform has usually bad decay properties, even for Schwartz class function

Once this is done, we can formulate and prove Theorem \ref{thmexteriorgene}. An abstract lemma shows that it is enough to describe the kernel of $\m 1_{|s| \ge R} \Td: L^2(\m R^d) \to L^2(\m R \times \m S^{d-1})$, and similarly for $\partial_s \Td$. The computation of both kernels is really similar, but has to be carried out separately: we concentrate on $\ker \m 1_{|s| \ge R} \Td$. The computation of this kernel follows from a combination of several observations. 

First, we can restrict to compute harmonics by harmonics, that is for function of the form $w(|x|) Y_\ell(x/|x|)$. Second, denoting $\q N_{\ell}^0$ this kernel, we can prove that its image by $\Td$ is actually a polynomial restricted to $|s| \le R$, with a bound on the degree. As $\Td$ is an isometry on $L^2$, we infer that $\q N_{\ell}^0$ is finite dimensional. Third, an important property is that $\q N_{\ell}^0$ is stable by a semi-group of dilations, from which we infer that it must be made of very specific function $w$, of type $w(r) = r^\alpha \ln^\beta(|r|)$. We prove that $\beta=0$ and that $\alpha$ has to be an integer as a consequence of a further stability property, namely by applying an operator related to the Laplacian $\Delta$ (correctly localized). Finally, we have to prove that all the remaining functions do actually belong to $\q N_{\ell}^0$. This does not follow in an obvious way by direct computations, because of integrability issues due to low decay; instead we use an induction and stability by derivation again.

\bigskip

The next sections are organized as follows. In section \ref{sectproofthm}, we prove Theorem \ref{th1} and Corollary \ref{prop:ext_energy}. As an application, we quickly deduce Propositions \ref{thmexteriorgeneR0}. In section 3, we detail the proof of Proposition \ref{prop:ortho_cutoff}. 
In section \ref{sectRadsob}, we develop a suitable functional framework for the Radon transform in Sobolev space and in section \ref{secRadball}, we  study the operator $\Td$ outside balls and prove Theorem \ref{thmexteriorgene}.

\section{Proof of Theorem \ref{th1} and consequences}
\label{sectproofthm}

Before we proceed with the main proofs, let us first observe that $\q T: L^2(\m R^d) \to L^2(\m R \times \m S^{d-1})$ and $\partial_s \q T: \dot H^1(\m R^d) \to L^2(\m R \times \m S^{d-1})$ are isometries.

\begin{lem} \label{lem:T_iso}
The operator $\q T: L^2(\m R^d) \to L^2(\m R \times \m S^{d-1})$ is a (well defined and) continuous map and
\[ \forall v \in L^2(\m R^d), \quad \| \q T v \|_{L^2(\m R \times \m S^{d-1})} = \| v \|_{L^2(\m R^d)}. \]
Similarly, $\partial_s \q T: \dot H^1 (\m R^d) \to L^2(\m R \times \m S^{d-1})$ is a (well defined and) continuous map and
\[ \forall v \in \dot H^1(\m R^d), \quad \| \partial_s \q T v \|_{L^2(\m R \times \m S^{d-1})} = \| v \|_{\dot H^1}. \]
Also, if $d$ is odd, then one has the symmetry: for $s \in \m R$ and $\omega \in \m S^{d-1}$,
\begin{align}
\label{eq:sym_T}
\q T v(-s,\omega) = (-1)^{\frac{d-1}{2}}\q T v(s,-\omega),
\end{align}
while if $d$ is even
\begin{align}
\label{eq:sym_Tpair}
\q T v(-s,\omega) = (-1)^{\frac{d}{2}}\mathcal{H}(\q T v)(s,-\omega). 
\end{align}
Above, $\q H$ denotes the Hilbert transform with respect to the $s$ variable.
\end{lem}

\begin{proof}
The point is that the Fourier multiplier defining $\q T$ is of modulus $1$ for all $\nu, \omega$. It suffices to show the equalities of norms for $v \in \q S(\m R^d)$.

Let us first prove the first statement: we compute via Plancherel on $\m R$ and $\m R^d$,
\begin{align*}
 \|  \q T v \|_{L^2(\m R \times \m S^{d-1})}^2 & = \frac{1}{2\pi} \int_{\m S^{d-1}} \int_{\m R} |\q F_{\m R} \q Tv(\nu,\omega)|^2 d\nu d\omega \\
& = \frac{1}{2\pi} c_0^2 \int_{\m S^{d-1}} \int_{\m R}  |\hat v(\nu\omega)|^2 |\nu|^{d-1} d\nu d\omega \\
& = \frac{1}{2\pi} \frac{1}{2 (2\pi)^{d-1}} 2  \int_{\m S^{d-1}} \int_{0}^{\infty}  |\hat v(\nu\omega)|^2 \nu^{d-1} d\nu d\omega \\
& = \frac{1}{(2\pi)^d} \int_{\m R^d} |\hat v(\xi)|^2 d\xi = \| v \|_{L^2(\m R^d)}^2.
\end{align*}
For the second statement, we observe that
\[ \q F_{\m R} (\partial_s \q Tv)(\nu,\omega) = i\nu \q F_{\m R} (\q Tv)(\nu,\omega), \]
so that the same computations give 
\begin{align*}
\MoveEqLeft \| \partial_s \q T v \|_{L^2(\m R \times \m S^{d-1})}^2 = \frac{1}{2\pi} \frac{1}{2 (2\pi)^{d-1}} 2  \int_{\m S^{d-1}} \int_{0}^{\infty}  |\hat v(\nu,\omega)|^2 \nu^{d+1} d\nu d\omega \\
& = \frac{1}{(2\pi)^d} \int_{\m R^d} |i \xi \hat v(\xi)|^2 d\xi =  \frac{1}{(2\pi)^d} \int_{\m R^d} |\q F_{\m R^d}(\nabla v) (\xi)|^2 d\xi = \| \nabla v \|_{L^2(\m R^d)}^2.
\end{align*} 
When $d$ is odd, observe that $e^{i\tau} = (-1)^{\frac{d-1}{2}} e^{-i\tau}$. Therefore, for $\nu \ne 0$,
\begin{align*}
\q F_{s \to \nu}( \q T v(-s,\omega))(\nu,\omega) &=  \q F_{\m R} (\q T v)(-\nu,\omega) \\
& = c_0 |\nu|^{\frac{d-1}{2}} ( e^{i\tau} \m 1_{-\nu < 0} + e^{-i\tau} \m 1_{-\nu \ge 0}) \hat v(-\nu \omega) \\
& =  (-1)^{\frac{d-1}{2}} c_0 |\nu|^{\frac{d-1}{2}} ( e^{i\tau} \m 1_{\nu \le 0} + e^{-i\tau} \m 1_{\nu > 0}) \hat v(\nu (-\omega))) \\
& = (-1)^{\frac{d-1}{2}}\q F_{\m R} (\q T v)(\nu,-\omega).
\end{align*}
When $d$ is even, we have instead $e^{i\tau} = -i(-1)^{\frac{d}{2}} e^{-i\tau}$. Therefore, for $\nu \ne 0$,
\begin{align*}
\q F_{s \to \nu}( \q T v(-s,\omega))(\nu,\omega) &=  \q F_{\m R} (\q T v)(-\nu,\omega) \\
& = c_0 |\nu|^{\frac{d-1}{2}} ( e^{i\tau} \m 1_{-\nu < 0} + e^{-i\tau} \m 1_{-\nu \ge 0}) \hat v(-\nu \omega) \\
& =  -i(-1)^{\frac{d}{2}} c_0 |\nu|^{\frac{d-1}{2}} ( -e^{i\tau} \m 1_{\nu \le 0} + e^{-i\tau} \m 1_{\nu > 0}) \hat v(\nu (-\omega))) \\
& = (-1)^{\frac{d}{2}}\q F_{\m R} (\mathcal{H}\q T v)(\nu,-\omega).
\end{align*}
We conclude in both cases by taking inverse Fourier transform in the $\nu$ variable.
\end{proof}

\begin{proof}[Proof of Theorem \ref{th1}]

We first prove 1), that is the computations for the half-wave equation.

We will first assume that $f\in \mathcal{S}(\R^d)$ are smooth and decaying, and that $ \hat{f} \in \q D(\R^d \setminus \{ 0 \})$ is smooth and has compact support away from 0.

We denote $v$ the solution of the first (outcoming) half wave equation, so that 
\[ \hat{v}(t,\xi)=e^{it|\xi|}\hat{f}(\xi). \]
 The inversion formula gives 
\begin{align*}
v(t,x)&=\frac{1}{(2\pi)^d}\int_{\R^d}e^{ix\cdot\xi}e^{it|\xi|}\hat{f}(\xi)d\xi\\
&=\frac{1}{(2\pi)^d}\int_{0}^{+\infty}\int_{\m S^{d-1}} e^{ir x\cdot\omega }e^{itr}\hat{f}(r\omega)r^{d-1}dr d\omega \\
&=\frac{1}{(2\pi)^d}\int_{0}^{+\infty}r^{d-1}e^{itr} \int_{\m S^{d-1}} e^{ir x\cdot\omega }\hat{f}(r\omega)dr d\omega.
\end{align*}
We will use the  polar coordinates notations $r_x$, $\omega_x$, that is: $r_x :=|x|$ and $\ds \omega_x:=\frac{x}{|x|}$.

We study the second integral thanks to the method of stationary phase, with $r$ fixed as a parameter that vary on a bounded set (relative to the support of $f$) and $r_x$ as a large parameter. For $r\in \R_{+}^{*}$ and $\sigma \in \m S^{d-1}$, we denote $\varphi_{r,\sigma}: \m S^{d-1} \to \m R$ the function defined by
\[ \varphi_{r,\sigma}(\omega)=r \sigma \cdot\omega. \]
The second integral can then be written
\[
\int_{\m S^{d-1}} e^{ir x\cdot\omega }\hat{f}(r\omega)d\omega = \int_{\m S^{d-1}} e^{i r_x \varphi_{r,\omega_x} (\omega)}\hat{f}(r\omega)d\omega.
\]
Observe that for all $\sigma \in \m S^{d-1}$, $\varphi_{r,\sigma}$ has two critical points
\begin{itemize}
\item $\omega_1=\sigma$ with $signature(\varphi_{r,\sigma}''(\omega_1))=(0,-(d-1))$ and $\det(\varphi_{r,\sigma}''(\omega_1))=(-r)^{d-1}$
\item $\omega_2=-\sigma$ with $signature(\varphi_{r,\sigma}''(\omega_2))=(d-1,0)$ and $\det(\varphi_{r,\sigma}''(\omega_2))=r^{d-1}$
\end{itemize}
 Note that the computations of the properties of $\varphi_{r,\sigma}''$ can be obtained for instance by reducing to $\sigma=(0,\dots,0,1)$ by rotation invariance and working in local coordinates $\omega=(x_{1},\dots,x_{d-1},\pm \sqrt{1- (x_{1}^{2}+\dots+x_{d-1}^{2})})$ close to $\pm \sigma$.
 
So, using the oscillatory integral formula, we have
\begin{gather}
\begin{aligned}
 \int_{\m S^{d-1}} e^{ir x\cdot\omega }\hat{f}(r\omega)d\omega & =\left(\frac{2\pi}{rr_x}\right)^{\frac{d-1}{2}}e^{-i\tau}e^{irr_x}\hat{f}(r\omega_x) \\
& \qquad +\left(\frac{2\pi}{rr_x}\right)^{\frac{d-1}{2}}e^{i\tau}e^{-irr_x}\hat{f}(-r\omega_x)+ \text{Rem}(r,r_x,\omega_{x}),
\end{aligned} \\
\label{estimR} \text{with} \quad |\text{Rem}(r,r_x,\omega_{x})| \le \frac{C}{r_x^{\frac{d+1}{2}}}.
\end{gather}
and $\text{Rem}$ has compact support in $\R_{+}^{*}$ as a function of $r$. We refer for instance to Grigis-Sj\"ostrand \cite[Proposition 2.3, p.22]{GrigisSjostrandBook} or \cite[Theorem 7.7.5]{HormanderI}). In these references, the estimates for the oscillatory integral are given for regular compactly supported functions on $\R^{d-1}$; it is easy to obtain the associated result on the compact manifold $\m S^{d-1}$ by working in coordinate charts. The constant $C$ then depends on $\varphi_{r,\sigma}$ and some ($L^{\infty}$) bounds on the derivatives of  $\omega \mapsto \hat{f}(r\omega)$. One can check that once $f$ such that $ \hat{f} \in \q D(\R^d \setminus \{ 0 \})$  is fixed, the constant $C$ in \eqref{estimR} can be made uniform in $r$ and $\omega_{x}$. We also notice that in the above references, the estimate is sometimes written for $r_{x}\geq 1$, but it is easy to check that it remains true for small $r_{x}$, for which it is actually trivial. We also refer to \cite[Theorem 7.7.14]{HormanderI} for a more geometric result on such integral on a hypersurface.

Therefore, we have the pointwise estimate of the error term
\begin{gather} \label{est:err}
\forall t \in \m R, x \in \m R^d, \quad  \left| \int_0^\infty \text{Rem}(r,r_x,\omega_x) r^{d-1} e^{itr} dr \right| \le \frac{C}{|x|^{\frac{d+1}{2}}}.
\end{gather}
We now compute the contribution of the other two terms
 \begin{align*}
\MoveEqLeft \tilde v(t,x) := \left(\frac{1}{r_x}\right)^{\frac{d-1}{2}}\frac{1}{(2\pi)^{\frac{d+1}{2}}} \\
& \times \int_{0}^{+\infty}r^{\frac{d-1}{2}}e^{itr} \left[e^{-i\tau}e^{irr_x}\hat{f}(r\omega_x)+e^{i\tau}e^{-irr_x}\hat{f}(-r\omega_x)\right] dr
\end{align*}
The first term writes
\begin{align*}
\tilde v_1(t,x)&: =\left(\frac{1}{r_x}\right)^{\frac{d-1}{2}}\frac{e^{-i\tau}}{(2\pi)^{\frac{d+1}{2}}}\int_{0}^{+\infty}r^{\frac{d-1}{2}}e^{i(t+r_x)r}\hat{f}(r\omega_x) dr.
\end{align*}
We want an asymptotic as $t \to +\infty$ so that $t+r_x$ is a large positive parameter. Therefore the phase in $r$ is never critical, and we get that for any $N\in \N$, there exists $C_N >0$ such that
\begin{gather}
\label{estimu1}
\forall t  \ge 0, \ \forall x \in \m R^d, \quad |\tilde v_1(t,x)|\le C_N \frac{1}{r_x^\frac{{d-1}}{2} (t+r_x)^N}
\end{gather}
Thus we are left with the second term
\begin{align*}
\tilde v_2(t,x)&:=\left(\frac{1}{r_x}\right)^{\frac{d-1}{2}}\frac{e^{i\tau}}{(2\pi)^{\frac{d+1}{2}}}\int_{0}^{+\infty}r^{\frac{d-1}{2}}e^{ir(t-r_x)}\hat{f}(-r\omega_x) dr\\
&=\left(\frac{1}{r_x}\right)^{\frac{d-1}{2}}\frac{e^{i\tau}}{(2\pi)^{\frac{d+1}{2}}}\int_{-\infty}^{0} |r|^{\frac{d-1}{2}}e^{-ir(t-r_x)}\hat{f}(r\omega_x) dr \\
& = \left(\frac{1}{r_x}\right)^{\frac{d-1}{2}}\frac{e^{i\tau }}{(2\pi)^{\frac{d-1}{2}}}f_{\omega_x}^- (r_x-t).
\end{align*}
(Recall that $f_\omega^-: \m R \to \m C$ has Fourier transform $\q F_{\m R}(f_\omega^-)(\nu) = \m 1_{\nu \le 0 } |\nu|^{\frac{d-1}{2}}\hat{f}(\nu\omega)$; note that it is a Schwartz function because the support of  its Fourier transform is away from zero).

Gathering our computations yields the pointwise estimate, for $t \ge 0$ and $x \in \m R^d$:
\begin{gather} \label{est:u_ptw}
\left| (e^{i t|D|} f)(x) - \frac{e^{i\tau}}{(2\pi |x|)^{\frac{d-1}{2}} }   f_{\omega_x}^- (|x|-t) \right| \le \frac{C}{|x|^{\frac{d+1}{2}}}.
\end{gather}
This will make it quite clear that the solution is concentrated close to the annulus 
\[ A_{t,R} := \left\{ x\in\R^d:  |t|-R \le |x| \le  |t| +R \right\} \]
 for large $R \ge 0$ as $t \to +\infty$ (with $|t| \ge R$). Due to the conservation of $L^2$ norm, we will infer that $e^{it D|} f$ has vanishing $L^2$ norm outside large annuli centered around the sphere of radius $|t|$.

Indeed, we have more precisely
\begin{align*}
\MoveEqLeft \| e^{it |D|} f \|_{L^2(A_{t,R})}^2 =  \int_{\omega \in \m S^{d-1}} \int_{r=t-R}^{t+R} |(e^{it |D|} f)(r\omega)|^2 r^{d-1} drd\omega \\
& =  \frac{1}{(2\pi)^{d-1}} \int_{\m S^{d-1}} \int_{r=t-R}^{t+R} \left( | f_{\omega}^-(r-t) |^2  +  | f_{\omega}^-(r-t) | O \left( \frac{1}{r} \right) + O \left( \frac{1}{r^2} \right) \right) dr d\omega \\
& =  \frac{1}{(2\pi)^{d-1}} \int_{\m S^{d-1}} \left( \int_{-R}^{R}  | f_{\omega}^-(r)|^2 dr + \| f_{\omega}^{-} \|_{L^2(\m R)} O \left(  \int_{t-R}^{t+R} \frac{dr}{r^2} \right)^{1/2} + O \left(  \int_{t-R}^{t+R} \frac{dr}{r^2} \right) |\right) d\omega\\
& = \frac{1}{(2\pi)^{d-1}} \int_{\m S^{d-1}} \int_{-R}^{R}  | f_{\omega}^-(r)|^2 dr d\omega + O\left( \frac{1}{\sqrt{t-R}} \right), 
\end{align*}
 where the implicit constant is uniform in $t \ge R \ge 0$; we used the Cauchy-Schwarz inequality, and the fact that, due to Plancherel identity
\begin{align*}
\int_{\m S^{d-1}} \int_{\m R}  | f_{\omega}^-(r)|^2 dr d\omega & = \frac{1}{2\pi} \int_{\m S^{d-1}} \int_{-\infty}^{0} |\hat f(\nu \omega)|^2 \nu^{d-1} d\nu = \frac{1}{2\pi} \int_{\m R^d} |\hat f(\xi)|^2 d\xi \\
& = (2\pi)^{d-1}  \| f \|_{L^2}^2.
\end{align*}
Let $\e>0$. The above computation shows that for $R$ large enough
\[ \left| \frac{1}{(2\pi)^{d-1}} \int_{\m S^{d-1}} \int_{-R}^{R}  | f_{\omega}^-(r)|^2 dr d\omega - \| f \|_{L^2}^2 \right| \le \e. \]
Therefore, for such $R$,
\[ \limsup_{t \to +\infty} | \| e^{it |D|} f \|_{L^2(A_{t,R})}^2 - \| f \|_{L^2}^2| \le \e, \]
As $\| e^{it |D|} f \|_{L^2} = \| f \|_{L^2}$, we get
\begin{align} \label{eq:w_out}
\lim_{R \to +\infty} \limsup_{t \to +\infty}  \| e^{it |D|} f \|_{L^2({}^c A_{t,R})} =0,
\end{align}
which is \eqref{eq:conc_cone} for $t\to +\infty$.

We can now finish up and prove \eqref{eq:expansion_w}. Due to the pointwise bound \eqref{est:u_ptw}, we have
\begin{align*}
\int_{|x| \ge t/2} \left| (e^{it|D|}f)(x) - \frac{e^{i\tau}}{(2\pi |x|)^{\frac{d-1}{2}} }  f_{\omega_x}^- (|x|-t) \right|^2 dx \le C \int_{t/2}^{+\infty} \frac{dr}{r^2}\underset{t\to +\infty}{\longrightarrow}  0.
\end{align*}
Now, from \eqref{eq:w_out}, one easily has that
\begin{gather} \label{uzerot2}\| e^{it|D|}f \|_{L^2(|x| \le t/2)} \underset{t\to +\infty}{\longrightarrow}  0. \end{gather}
On the other hand, 
\begin{align}
\nonumber
\MoveEqLeft \left\| \frac{1}{|x|^{\frac{d-1}{2}}} f_{\omega_x}^- (|x|-t) \right\|_{L^2(|x| \le t/2)}^2  = \int_{B(0,t/2)} \frac{1}{|x|^{d-1}} |f_{\omega_x}^- (|x|-t)|^2 dx \\
& = \int_{r=0}^{t/2} \int_{\omega \in \m S^{d-1}} |f_{\omega}^- (r-t)|^2 dr d\omega  =  \int_{\omega \in \m S^{d-1}} \int_{r=-t}^{-t/2} |f_{\omega}^-(r)|^2 dr d\omega.\label{calculL2f}
\end{align}
We already saw that
\[  \int_{\omega \in \m S^{d-1}} \int_{\m R} |f_{\omega}^-(r)|^2 dr d\omega = \| f \|_{L^2}^2 <+\infty, \]
so that the above is an exhausting integral, which thus tends to $0$ as $t \to +\infty$. We infer from this and \eqref{uzerot2} that
\[ \int_{|x| \le t/2} \left| (e^{it|D|}f)(x) - \frac{e^{i\tau}}{(2\pi |x|)^{\frac{d-1}{2}} }  f_{\omega_x}^- (|x|-t) \right|^2 dx \underset{t\to +\infty}{\longrightarrow}  0. \]
Hence \eqref{eq:expansion_w} is proved and 1) is complete for the case $ \hat{f} \in \q D(\R^d \setminus \{ 0 \})$. 

For the general case, by density, it is sufficient to notice that for fixed $t$, the maps 
\[  f\mapsto e^{it|D|} f \quad \text{and} \quad f\mapsto \left( x \mapsto  \frac{e^{i\tau}}{(2\pi|x|)^{\frac{d-1}{2}}} f_{x/|x|}^- (|x|-t)) \right) \]
are linear continuous from $L^{2}(\R^{d})$ to $L^{2}(\R^{d})$ with bound uniform in $t$. The first one is obvious due to Plancherel while the second one can be obtained by a computation similar to \eqref{calculL2f}.

\bigskip

Before we prove, 2), let us first derive the formula \eqref{eq:expansion_g} for $w(t) = e^{-it |D|} g$. One can proceed as before, by noticing that  that up to an error term with size as in \eqref{est:err}, the main contribution is 
\begin{align*}
\MoveEqLeft \tilde w(t,x)= \left(\frac{1}{r_x}\right)^{\frac{d-1}{2}}\frac{1}{(2\pi)^{\frac{d+1}{2}}} \\
& \times \int_{0}^{+\infty}r^{\frac{d-1}{2}}e^{-itr} \left[e^{-i\tau} e^{irr_x}\hat{g}(r\omega_x)+e^{i\tau}e^{-irr_x}\hat{g}(-r\omega_x)\right],
\end{align*}
and that this time, the only relevant term is
\begin{align*}
\tilde w_1(t,x)&=\left(\frac{1}{r_x}\right)^{\frac{d-1}{2}}\frac{e^{-i\tau}}{(2\pi)^{\frac{d+1}{2}}}\int_{0}^{+\infty}r^{\frac{d-1}{2}}e^{-i(t-r_x)r}\hat{g}(r\omega_x) dr \\
& =\left(\frac{1}{r_x}\right)^{\frac{d-1}{2}}\frac{e^{-i \tau} }{(2\pi)^{\frac{d-1}{2}}} g_{\omega_x}^+(r_x-t),
\end{align*}
as $\q F_{\m R}(g_\omega^+)(r) = \m 1_{r \ge 0 }r^{\frac{d-1}{2}}\hat{g}(r\omega)$.

Or as mentioned in the introduction, one can also simply take complex conjugate in the expansion of $e^{it|D|} \overline{g}$ and observe that
\[ \overline{(\overline g)_\omega^-(s)} = g_\omega^+(s). \]
Indeed, taking Fourier transform, there hold
\begin{align*}
\q F_{\m R} ( \overline{(\overline g)_\omega^-})(\nu) & = \int_{\m R} e^{-i s \nu} \overline{(\overline g)_\omega^-(s)} ds = \overline{ \int_{\m R} e^{i s \nu} (\overline g)_\omega^-(s) ds} = \overline{\q F_{s \to \nu} ((\overline g)_\omega^-)(-\nu)} \\
& = \m 1_{- \nu \le 0} |\nu|^{\frac{d-1}{2}} \overline{ \hat {\overline g}(-\nu \omega)} =  \m 1_{\nu \ge 0} |\nu|^{\frac{d-1}{2}} \hat g(\nu \omega) =  \q F_{\m R} (g_\omega^+)(\nu).
 \end{align*}

\bigskip

We now turn to 2). 
We recall that with 
\[ f = \ds \frac{1}{2}  \left( u(0) + \frac{1}{i |D|} \partial_t u(0) \right) \quad \text{and} \quad \ds g= \frac{1}{2}  \left( u(0) - \frac{1}{i |D|} \partial_t u(0) \right), \]
we have $f,g \in \dot H^1(\m R^d)$ and
\begin{gather} \label{eq:du_f_g} \nabla_{t,x} u(t) = e^{i t |D|} \begin{pmatrix}
i |D| f \\
\nabla_x f  \end{pmatrix} + e^{-i t |D|} \begin{pmatrix}
-i |D| g\\
\nabla_x g \end{pmatrix} \in L^2(\m R^d,\m C^{1+d}).
\end{gather}
So that for $i=1, \dots, d$, and using \eqref{eq:expansion_w} and \eqref{eq:expansion_g}
\begin{align} \label{eq:dtu}
 \partial_t u(t) & = \frac{1}{(2\pi |x|)^{\frac{d-1}{2}}} \left( e^{i\tau} (i |D| f)_{x/|x|}^- + e^{-i\tau} (-i |D| g)_{x/|x|}^+ \right)(|x|-t) + \e_0(t,x), \\
 \partial_i u(t) & = \frac{1}{(2\pi |x|)^{\frac{d-1}{2}}} \left( e^{i\tau} (\partial_i f)_{x/|x|}^- + e^{-i\tau} (\partial_i g)_{x/|x|}^+ \right)(|x|-t) + \e_i(t,x) \label{eq:diu}
 \end{align} 
where $\| \e_i(t) \|_{L^2(\m R^d)} \to 0$ as $t \to +\infty$. Now we have 
\begin{gather*}
\q F_{\m R} (\partial_i f)_{\omega}^-(\nu)  = \m 1_{\nu < 0} \nu^{\frac{d-1}{2}} \widehat{\partial_i f}(\nu \omega) = \m 1_{\nu < 0} |\nu|^{\frac{d-1}{2}} (i \nu \omega_i) \hat f(\nu \omega) = \omega_i \q F_{\m R}( \partial_\rho f_\omega^-)(\nu),
\end{gather*}
so that 
\begin{align*}
(\partial_i f)_{\omega}^-(s) = \omega_i \partial_s (f_\omega^-)(s) \quad \text{and similarly} \quad  (\partial_i g)_{\omega}^+(s) = \omega_i \partial_s (g_\omega^+)(s). 
\end{align*}
Regarding the time derivatives:
\begin{gather}
\q F_{\m R} (i|D| f)_{\omega}^-(\nu) = \m 1_{\nu < 0} |\nu|^{\frac{d-1}{2}} \widehat{i|D| f}(\nu \omega) = \m 1_{\nu < 0} |\nu|^{\frac{d-1}{2}} (-i\nu) \hat f(\nu \omega) = - \q F_{\m R} (\partial_\rho f_\omega^-)(\nu)
\end{gather}
so that 
\[ (i|D| f)_{\omega}^- = - \partial_s f_\omega^- \quad \text{and similarly} \quad (-i|D| g)_{\omega}^+ = - \partial_s g_\omega^+. \]
This can be sumarized by considering  the function defined for $\omega \in \m S^{d-1}$ and $s \in \m R$ by
\[ h(s, \omega) := \partial_s (e^{i \tau}f_{\omega}^-  + e^{-i\tau} g_{\omega}^+)(s), \]
so that
\begin{gather} \label{eq:u_h} \nabla_{t,x} u(t,x) =   \frac{1}{(2\pi |x|)^{\frac{d-1}{2}}}  h \left( |x|-t, \frac{x}{|x|} \right) \begin{pmatrix} 
-1 \\
x/|x|
\end{pmatrix}
+ \e(t,x),
\end{gather}
where $\e(t) \to 0$ in $L^2(\m R^d,\m C^{1+d})$. It suffices to relate $h$ and $\q T$, which we do now by computing the 1D Fourier transform of $h$ in the $s$ variable:
\begin{align*}
\q F_{s \to \nu} h(\nu,\omega) & = i \nu(e^{i \tau} \q F_{\m R} f_{\omega}^-  + e^{-i\tau} \q F_{\m R}(g_{\omega}^+)(\nu) \\
& = i \nu  |\nu|^{\frac{d-1}{2}}  (e^{i \tau}  \m 1_{\nu \le 0} \hat f + e^{-i \tau}  \m 1_{\nu \ge 0} \hat g)(\nu \omega) \\
& = i \nu  |\nu|^{\frac{d-1}{2}}  \left[ e^{i \tau}  \m 1_{\nu \le 0} \frac{1}{2} \left( \hat u_0 + \frac{1}{i|\nu|} \hat u_1 \right) + e^{-i \tau}  \m 1_{\nu \ge 0}  \frac{1}{2} \left( \hat u_0 - \frac{1}{i|\nu|} \hat u_1 \right)  \right] (\nu \omega) \\
& = \frac{1}{2} i\nu  |\nu|^{\frac{d-1}{2}}  \left[ ( e^{i \tau}  \m 1_{\nu \le 0} +  e^{-i \tau}  \m 1_{\nu \ge 0})  \hat u_0 + \frac{1}{i|\nu|} ( e^{i \tau}  \m 1_{\nu \le 0} -  e^{-i \tau}  \m 1_{\nu \ge 0})  \hat u_1 \right] (\nu \omega) \\
& = \frac{1}{2}  |\nu|^{\frac{d-1}{2}} \left[ i\nu ( e^{i \tau}  \m 1_{\nu \le 0} +  e^{-i \tau}  \m 1_{\nu \ge 0})  \hat u_0 -  ( e^{i \tau}  \m 1_{\nu \le 0} +  e^{-i \tau}  \m 1_{\nu \ge 0})  \hat u_1 \right](\nu \omega) \\
& =  \frac{1}{2c_{0}} \left( i\nu \q F_{s \to \nu} (\q Tu_0)(\nu,\omega) - \q F_{s \to \nu} (\q Tu_1)(\nu,\omega) \right)\\
& = \q F_{s \to \nu} \left( \frac{1}{2 c_0} (\partial_s \q Tu_0 - \q Tu_1)(\nu,\omega) \right).
\end{align*}
Via inverse Fourier transform, we get $h(s, \omega) = \frac{1}{2 c_0}  (\partial_s \q Tu_0 - \q Tu_1)$, which is \eqref{eq:fg_T}, 
and from \eqref{eq:u_h}, we derive \eqref{eq:expansion_w_1}. \eqref{eq:conc_cone_2} follows similarly as for the half-wave case.
\end{proof}
\begin{remark}
Performing similar computations in the case $u_{0}\in L^{2}(\R^{d})$ and $u_{1}=0$, which is $f=g=u_{0}/2$, we can write \eqref{equivfinallemma} in a simplified form, namely
\begin{align}
\MoveEqLeft \lim_{t \to +\infty} \nor{u}{L^2(|x| \ge t+R)}^2  = \frac{1}{2}\int_{\omega \in \m S^{d-1}}\nor{\Td u_{0}}{L^2([R,+\infty))}^2  d\omega. \label{equivfinallemmaT}
\end{align}
\end{remark}
\begin{remark}
Note that it could seem surprising at first that from estimates like \eqref{estimR} where the constants $C$ is strongly dependent on the smooth function $f$ and some of its derivatives, we can deduce some uniform estimates like \eqref{eq:expansion_w} for any $L^{2}$ functions. It should be noticed then that the stationary phase estimates that we use are then combined with $L^{2}$ estimates. They actually prove that the main term that we obtain contains all the $L^{2}$ norm.
\end{remark}
\begin{proof}[Proof of Corollary \ref{prop:ext_energy}]
Now, we turn to the proof of \eqref{equivfinallemma}, that is the computation of the $L^2$ norm outside the ball $B _{t+R}= \left\{x\in\R^d; r_x < t+R\right\}$.
From \eqref{eq:expansion_w} and \eqref{eq:expansion_g}
\[ u(t,x) = (e^{it|D|} f + e^{-it|D|} g)(x) = \frac{1}{(2\pi |x|)^{\frac{d-1}{2}}} \left( e^{i \tau} f_{x/|x|}^- +  e^{-i \tau} g_{x/|x|}^+ \right)(|x|-t) + o_{L^2}(1). \]
The same computations as before gives
\[ \left\| \frac{1}{(2\pi |x|)^{\frac{d-1}{2}}}  f_{x/|x|}^- \right\|_{L^2} = \| f \|_{L^2}, \]
(and the same for $g$, so that, for $t \ge 0$,
\begin{align*}
\MoveEqLeft (2\pi)^{d-1} \| u(t) \|_{L^2(|x| \ge t+R)}^2  = \int_{\omega \in \m S^{d-1}} \int_{t+R}^{+\infty} |u(t,r\omega)|^2 r^{d-1} dr d\omega \\
& = \int_{\omega \in \m S^{d-1}} \int_{t+R}^{+\infty} |e^{i\tau} f_{\omega}^- (r-t) + e^{-i\tau} g_{\omega}^+(r-t)|^2  dr d\omega + o_{L^2}(1) \\
& = \int_{\omega \in \m S^{d-1}} \int_{R}^{+\infty} |e^{i\tau} f_{\omega}^- (r) + e^{-i\tau} g_{\omega}^+(r)|^2  dr d\omega + o_{L^2}(1),
\end{align*}
as desired for the $L^2$ case.

In order to complete the energy space case, we invoke \eqref{eq:expansion_w_1}. As $\partial_s \q Tu_0 - \q Tu_1 \in L^2(\m R \times \m S^{d-1})$, we get as before
\begin{align*}
\MoveEqLeft
2 \| \partial_t u(t) \|_{L^2(|x| \ge t+R)}^2 =  \int_{\omega \in \m S^{d-1}} \int_{t+R}^{+\infty} |\partial_s \q Tu_0 - \q Tu_1)(r-t,\omega) |^2  dr d\omega + o_{L^2}(1) \\
& = \| \partial_s \q Tu_0 - \q Tu_1 \|_{L^2([R,+\infty) \times \m S^{d-1})}^2 , 
\end{align*}
and an analogous computation for $\| \nabla u(t) \|_{L^2(|x| \ge t+R)}^2$: as $|x/|x|| =|-1|$, both limits are equal. \eqref{equivfinallemmaH1} and \eqref{equivfinallemmaH1_2} are proved.

Finally, for \eqref{eq:ortho_u0_u1}, it suffices to notice that $t \mapsto u(-t)$ is the solution to the wave equation with initial data $(u_0,-u_1)$, so by linearity of $\q T$,
\begin{align*}
\MoveEqLeft \lim_{t \to -\infty} \| \partial_t u(t) \|_{L^2(|x| \ge |t|+R)}^2 =  \lim_{t \to +\infty}  \| \partial_t u(-t) \|_{L^2(|x| \ge |t|+R)}^2 \\
& = \frac{1}{2} \| \partial_s \q Tu_0 - \q T(-u_1) \|_{L^2([R,+\infty) \times \m S^{d-1})}  = \frac{1}{2} \| \partial_s \q Tu_0 + \q Tu_1 \|_{L^2([R,+\infty) \times \m S^{d-1})}^2,
\end{align*}
and the same holds for $\nabla_x u$. Therefore, expanding the squares, 
\begin{align*}
E_{ext,R}(u) & = \frac{1}{2} ( \| \partial_s \q Tu_0 - \q Tu_1 \|_{L^2([R,+\infty) \times \m S^{d-1})}^2 +  \| \partial_s \q Tu_0 + \q Tu_1 \|_{L^2([R,+\infty) \times \m S^{d-1})}^2) \\
& =  \| \partial_s \q Tu_0 \|_{L^2([R,+\infty) \times \m S^{d-1})}^2 + \|  \q Tu_1 \|_{L^2([R,+\infty) \times \m S^{d-1})}^2. \qedhere
\end{align*}
\end{proof}

\begin{proof}[Proof of Proposition \ref{thmexteriorgeneR0}]
The goal is to compute $E_{ext,0}(u)$ when the dimension $d$ is odd. In that case, the symmetry \eqref{eq:sym_T} is available, so that 
\begin{align*} 
\MoveEqLeft 2 \| \Td u_1 \|_{L^2([0,+\infty) \times \m S^{d-1})}^2  = 2 \int_{\m S^{d-1}} \int_0^{+\infty} |\Td u_1(s,\omega)|^2 ds d\omega \\
& = \int_{\m S^{d-1}} \int_0^{+\infty} |\Td u_1(s,\omega)|^2 ds d\omega+ \int_{\m S^{d-1}} \int_0^{+\infty} |\Td u_1(-s,-\omega)|^2) ds d\omega \\
& = \int_{\m S^{d-1}} \int_{\m R} |\Td u_1(s,\omega)|^2 dsd\omega = \| \Td u_1 \|_{L^2(\m R \times \m S^{d-1})}^2 = \| u_1 \|_{L^2}^2 ,
\end{align*}
due to the first part of Lemma \ref{lem:T_iso}. As one also has the symmetry $\partial_s \Td u_0(-s,\omega) = (-1)^{\frac{d+1}{2}} \partial_s \Td u_0(s,-\omega)$ (by differentiating \eqref{eq:sym_T}), similar computations show that
\[ 2 \| \partial_s \Td u_0 \|_{L^2([0,+\infty) \times \m S^{d-1})}^2 = \| \partial_s \Td u_0 \|_{L^2(\m R \times \m S^{d-1})}^2 = \| u_0 \|_{\dot H^1}^2, \]
where we used also the second part of Lemma  \ref{lem:T_iso}. Summing up and using \eqref{eq:ortho_u0_u1}, we conclude
\[
E_{ext,0}(u) =  \| \partial_s \Td u_0 \|_{L^2([0,+\infty) \times \m S^{d-1})}^2 +  \| \Td u_1 \|_{L^2([0,+\infty) \times \m S^{d-1})}^2 = \frac{1}{2} ( \| u_0 \|_{\dot H^1}^2 +  \| u_1 \|_{L^2}^2). \qedhere
\]
\end{proof}

To conclude this section, our goal is now to give an expression of the energy outside the light cone in even dimension, so as to prove Proposition \ref{extEnergy_even}. We adopt the following convention for the Hilbert transform $\mathcal{H}$ on the real line: for $f\in \mathcal{S}(\R)$, we denote
\begin{align*}
\mathcal{H} f(s)=p.v. \frac{1}{\pi}\int_{\R}\frac{f(r)}{s-r}dr \quad \text{so that} \quad
i\widehat{\mathcal{H} f}(\xi)=\sgn(\xi)\widehat{ f}(\xi),
\end{align*}
where $\sgn$ denotes the signum function. Also, for functions defined for $(s, \omega) \in \m R \times \m S^{d-1}$, $\q H$ denote the Hilbert transform with respect to the $s$ variable. 

We start with a lemma, for which it is convenient to recall the Hankel transform, defined for $f\in \q D(\m R_+)$ by
\[ Hf(s) = \int_{0}^{\infty} \frac{f(r)}{s+r}dr \]
$H$ extends to a bounded operator on $L^2(\m R_+)$ with norm $\pi$.

\begin{lem}
\label{lmfpmHilbHank}
Let $f\in L^{2}(\R)$ then
\begin{align*}
\| f\|_{L^2(\m R^+)}^2 +   \| \mathcal{H}f \|_{L^2(\m R^-)}^{2} & = \| f\|_{L^2}^2 +  \frac{1}{\pi}\Im \int_0^\infty \int_0^\infty \frac{\overline{\hat f (s)} \hat f(-r)}{r+s}drds\\
& = \frac{1}{2\pi}\| \hat f\|_{L^2}^2  +  \frac{1}{\pi^{2}}\Im \int_0^\infty \int_0^\infty \frac{\overline{\hat f (s)} \hat f(-r)}{r+s}drds\\
&=\frac{1}{2\pi}\| \hat f\|_{L^2}^2  -  \frac{1}{\pi^{2}}\Im  \int_0^\infty (H\widehat{\overline{f}}) (r) \hat f(r)dr,
\end{align*}

Moreover, for $f,g \in L^{2}(\R)$, we have
\begin{align*}
\left<f ,g \right>_{\R_{+}}-\left< \mathcal{H}f , \mathcal{H}g \right>_{\R_{-}}&=&\frac{1}{2i\pi^{2}}\left( \int_0^\infty \int_0^\infty \frac{\hat f (r) \overline{\hat g(s)}}{s-r}drds-\int_0^\infty \int_0^\infty  \frac{\hat f (-r) \overline{\hat g(-s)}}{s-r}drds\right).
\end{align*}\end{lem}

\begin{proof}
Denote $f^{+}= 1_{\R_{+}}(D)f$ and $f^{-}= 1_{\R_{-}}(D)f$, then
\[ f=f^{-}+f^{+} \quad \text{and} \quad i\mathcal{H}f=-f^{-}+f^{+}. \]
Therefore,
\begin{align*}
\| f \|_{L^2(\m R^+)}^2 & = \| f^- \|_{L^2(\m R^+)}^2 + \| f^+ \|_{L^2(\m R^+)}^2 +2 \Re  \int_{\m R^+} f^-  \overline{f^+} \\
\| \mathcal{H}f \|_{L^2(\m R^-)}^2 & = \| f^+ \|_{L^2(\m R^-)}^2 + \| f^- \|_{L^2(\m R^-)}^2 - 2 \Re \int_{\m R^-} f^-  \overline{f^+} \\ 
\| f\|_{L^2(\m R^+)}^2 +   \| \mathcal{H}f \|_{L^2(\m R^-)}^{2} & = \| f\|_{L^2}^2  + 2\Re \langle \sgn \cdot f^- ,  f^+ \rangle 
\end{align*}
We denote $ \langle f , g  \rangle=\int_{\m R} f  \overline{g}$ the standard (complex) scalar product.

Using Parseval formula, we get $ \langle \sgn \cdot f^- ,  f^+ \rangle   =\frac{1}{2\pi}\langle \widehat{  \sgn \cdot f^-} , \widehat{f^+} \rangle $. Moreover, recall that
\[ \widehat{  \sgn\cdot g}=\frac{1}{2\pi}\widehat{  \sgn }*\widehat{ g}=\frac{1}{\pi}\frac{1}{i\xi}*\widehat{ g}=\frac{1}{i} \q H \widehat{ g} \]
where $\q H$ is the $\m R$-Hilbert transform. Now
\[  \q H \widehat{f^-} (s) =\frac{1}{\pi} \int_{-\infty}^0 \frac{\widehat{ f}(r)}{s-r} dr = \frac{1}{\pi}  \int_0^{\infty} \frac{\widehat{f}(-r) }{s+r} dr=\frac{1}{\pi}  H \widehat{f}(-\cdot), \]
so that
\[ 2\Re \langle \sgn \cdot f^- ,  f^+ \rangle = \frac{1}{\pi^{2}}\Im  \int_0^\infty \int_0^\infty \frac{\overline{\hat f (s)} \hat f(-r)}{r+s}drds. \]
Concerning the crossed term
\begin{align*}
\MoveEqLeft \left<f ,g \right>_{\R_{+}}-\left< \mathcal{H}f , \mathcal{H}g \right>_{\R_{-}}\\
& = \left<f^-+ f^+ ,g^- + g^+ \right>_{\R_{+}}-\left<f^+-f^- ,g^+-g^- \right>_{\R_{-}}\\
& = \left<f^-, g^+  \right>_{\R_{+}} +  \left< f^- ,g^+ \right>_{\R_{-}} + \left<f^+, g^+  \right>_{\R_{+}} - \left<f^+, g^+  \right>_{\R_{-}} \\
& \qquad \left< f^+ , g^- \right>_{\R_{+}}+\left<f^+  ,g^- \right>_{\R_{-}} + \left<f^- ,g^- \right>_{\R_{+}} - \left<f^- ,g^- \right>_{\R_{-}} \\
& = \left<f^-, g^+  \right> + \left< \sgn  \cdot f^+, g^+  \right> + \left< f^+ , g^- \right> + \left< \sgn \cdot f^- ,g^- \right> \\
& = \left< \sgn \cdot f^+, g^+  \right> + \left< \sgn \cdot f^- ,g^- \right>.
\end{align*}
In the computations above, we used the support properties of the functions $f^{\pm}$ and $g^{\pm}$ in Fourier space. As before, using Parseval Theorem, we get $ \langle \sgn \cdot f^+ ,  f^+ \rangle   =\frac{1}{2\pi}\langle \widehat{  \sgn \cdot f^+} , \widehat{f^+} \rangle $ and using again $\widehat{  \sgn \cdot   f^{\pm}}=\frac{1}{i} \q H \widehat{ f^{\pm}}$, we get
\begin{align*}
\MoveEqLeft \left<f ,g \right>_{\R_{+}}-\left< \mathcal{H}f , \mathcal{H}g \right>_{\R_{-}} = \frac{1}{2i\pi^{2}}\left( \int_0^\infty \int_0^\infty \frac{\hat f (r) \overline{\hat g(s)}}{s-r}drds+\int_{-\infty}^0\int_{-\infty}^0 \frac{\hat f (r) \overline{\hat g(s)}}{s-r}drds\right)\\
&= \frac{1}{2i\pi^{2}}\left( \int_0^\infty \int_0^\infty \frac{\hat f (r) \overline{\hat g(s)}}{s-r}drds-\int_0^\infty \int_0^\infty  \frac{\hat f (-r) \overline{\hat g(-s)}}{s-r}drds\right). \qedhere
\end{align*}
\end{proof}

\begin{proof}[Proof of Proposition \ref{extEnergy_even}]

We start with \eqref{equivfinallemmaH1_2}, and use the change of variable $\omega\leftrightarrow -\omega$ and \eqref{eq:sym_Tpair}, to compute 
\begin{multline} \label{eq:even_Eext}
4 \lim_{t \to +\infty} \nor{\nabla u}{L^2(|x| \ge t)}^2 =  \|  \partial_{s}\Td u_{0} - \Td u_{1} \|_{L^2(\R_{+} \times \m S^{d-1})}^2
+  \|  \mathcal{ H} \partial_{s}\Td u_{0} +  \mathcal{ H}\Td u_{1} \|_{L^2(\R_{-}\times \m S^{d-1})}^2\\
=  \|  \partial_{s}\Td u_{0}\|_{L^2(\R_{+} \times \m S^{d-1})}^2
+ \|  \mathcal{ H} \partial_{s}\Td u_{0}\|_{L^2(\R_{-}\times \m S^{d-1})}^2\\
+ \| \Td u_{1} \|_{L^2(\R_{+} \times \m S^{d-1})}^2
+ \|  \mathcal{ H}\Td u_{1} \|_{L^2(\R_{-}\times \m S^{d-1})}^2\\
- 2 \Re \left< \partial_{s}\Td u_{0} ,\Td u_{1}\right>_{L^2(\R_{+} \times \m S^{d-1})}+ 2\Re\left< \mathcal{ H} \partial_{s}\Td u_{0} , \mathcal{ H}\Td u_{1}\right>_{L^2(\R_{-} \times \m S^{d-1})}.
\end{multline}
Let's give an expression for each of the 3 lines of the last equality above.
Recall \eqref{def:T} and observe that $e^{2i\tau}=-i(-1)^{\frac{d}{2}}$: then, for fixed $\omega$, we use the first part of Lemma \ref{lmfpmHilbHank} with $f$ such that $\hat{f}(\nu)= c_0  |\nu|^{\frac{d-1}{2}} (e^{i \tau}  \m1 _{\nu < 0}   + e^{-i\tau} \m 1_{\nu \ge 0} )\widehat{ u_{1}}(\nu \omega)$. This yields, for the $u_1$ terms (2nd line),
\begin{align*}
&\| \Td u_{1}\|_{L^2(\R_{+} \times \m S^{d-1})}^2
+\|  \mathcal{ H}\Td u_{1}\|_{L^2(\R_{-}\times \m S^{d-1})}^2\\
&=\| \Td u_{1}\|_{L^2(\R\times \m S^{d-1})}^2 + \Im  \frac{1}{2(2\pi)^{d-1}\pi^{2}} e^{2i\tau} \int_{\omega \in \m S^{d-1}} \int_0^\infty \int_0^\infty (rs)^{\frac{d-1}{2}}  \frac{\overline{ \widehat{ u_{1}}(s \omega)} \widehat{ u_{1}}(-r \omega)}{r+s}drds d\omega\\
&= \nor{u_{1}}{L^{2}}^{2}-   \frac{2(-1)^{\frac{d}{2}} }{(2\pi)^{d+1}}\Re\int_{\omega \in \m S^{d-1}} \int_0^\infty \int_0^\infty (rs)^{\frac{d-1}{2}}  \frac{\overline{ \widehat{ u_{1}}(s \omega)}  \widehat{ u_{1}}(-r \omega)}{r+s}drds d\omega.
\end{align*}
For the $u_0$ terms (1st line), we use now $f$ such that  $\hat{f}(\nu)= c_0 i\nu |\nu|^{\frac{d-1}{2}} (e^{i \tau}  \m1 _{\nu < 0}   + e^{-i\tau} \m 1_{\nu \ge 0} )\widehat{ u_{0}}(\nu \omega)$, and thie gives
\begin{align*}
\MoveEqLeft \|  \partial_{s}\Td u_{0}\|_{L^2(\R_{+} \times \m S^{d-1})}^2 +\|  \mathcal{ H} \partial_{s}\Td u_{0}\|_{L^2(\R_{-}\times \m S^{d-1})}^2\\
&=\|  \partial_{s}\Td u_{0}\|_{L^2(\R\times \m S^{d-1})}^2 - \Im \frac{1}{2(2\pi)^{d-1}\pi^{2}}  e^{2i\tau} \int_{\omega \in \m S^{d-1}} \int_0^\infty \int_0^\infty (rs)^{\frac{d+1}{2}}  \frac{\overline{ \widehat{ u_{0}}(s \omega)} \widehat{ u_{0}}(-r \omega)}{r+s}drds d\omega\\
&= \nor{u_{0}}{\dot{H}^{1}}^{2}+   \frac{2(-1)^{\frac{d}{2}} }{(2\pi)^{d+1}}\Re\int_{\omega \in \m S^{d-1}} \int_0^\infty \int_0^\infty (rs)^{\frac{d+1}{2}}  \frac{\overline{ \widehat{ u_{0}}(s \omega)}  \widehat{ u_{0}}(-r \omega)}{r+s}drds d\omega. 
\end{align*}
We now work on crossed terms (the last line of \eqref{eq:even_Eext}): for this, we use the second part of Lemma \ref{lmfpmHilbHank} with $f$ and $g$ such that 
\begin{align*}
\hat{f}(\nu) &= c_0 i\nu |\nu|^{\frac{d-1}{2}} (e^{i \tau}  \m1 _{\nu < 0}   + e^{-i\tau} \m 1_{\nu \ge 0} )\widehat{ u_{0}}(\nu \omega) \\
 \text{and} \quad \hat{g}(\nu) &= c_0  |\nu|^{\frac{d-1}{2}} (e^{i \tau}  \m1 _{\nu < 0}   + e^{-i\tau} \m 1_{\nu \ge 0} )\widehat{ u_{1}}(\nu \omega).
\end{align*}
We obtain
\begin{align*}
\MoveEqLeft -\Re \left< \partial_{s}\Td u_{0} ,\Td u_{1}\right>_{L^2(\R_{+} \times \m S^{d-1})}+\Re\left< \mathcal{ H} \partial_{s}\Td u_{0} , \mathcal{ H}\Td u_{1}\right>_{L^2(\R_{-} \times \m S^{d-1})}\\
& = - \Re\frac{1}{2(2\pi)^{d-1}}\frac{1}{2i\pi^{2}} i\int_{\omega \in \m S^{d-1}}\left(\int_0^\infty \int_0^\infty \frac{r^{\frac{d+1}{2}}\widehat{ u_{0}}(r \omega) s^{\frac{d-1}{2}}\overline{\widehat{ u_{1}}(s \omega)}}{s-r}drds \right. \\
& \qquad \left. + \int_0^\infty \int_0^\infty \frac{r^{\frac{d+1}{2}}\widehat{ u_{0}}(-r \omega) s^{\frac{d-1}{2}}\overline{\widehat{ u_{1}}(-s \omega)}}{s-r}drds\right) d\omega \\
& = - \Re\frac{2}{(2\pi)^{d+1}} \int_{\omega \in \m S^{d-1}}\int_0^\infty \int_0^\infty \frac{r^{\frac{d+1}{2}}\widehat{ u_{0}}(r \omega) s^{\frac{d-1}{2}}\overline{\widehat{ u_{1}}(s \omega)}}{s-r}drds d\omega.
\end{align*}
Summing up the three above expressions yields the desired identity.
\end{proof}

\section{Proof of Proposition \ref{prop:ortho_cutoff}}

In this section, we focus on the proof of Proposition \ref{prop:ortho_cutoff}. When expanding the decomposition of $u$ in order to get \eqref{eq:ortho_cutoff}, we are left with the cross terms: the main point is to show that these cross terms tend to 0. This is the purpose of the following lemma.

\begin{lem} \label{lem:wH1_conv_cut}
Let $\vec u = (u, \partial_t u)$ and, for $n \in \m N$,  $\vec w_n = (w_n, \partial_t w_n)$ be solutions to the linear wave equation \eqref{eq:lw}, bounded in $\q C(\m R, \dot H^1 \times L^2(\m R^d))$. Let $t_n \in \m R$, $x_n \in \m R^d$ and $r_n >0$ be three sequences.  Assume that that $\vec w_n(-t_n) \weak 0$ in $\dot H^1 \times L^2(\m R^d)$. Then
\begin{gather} \label{eq:wH1_cut}
\int_{|x-x_n| \ge r_n} \overline{\nabla_{t,x} w_{n}(0,x)} \cdot  \nabla_{t,x} u(t_n,x)  dx \to 0 \quad \text{as} \quad n \to +\infty.
\end{gather}
\end{lem}

\begin{proof}
We denote $x_n = \rho_n \omega_n$, where $\rho_n >0$ and $\omega_n \in \m S^{d-1}$. It is enough to prove that for any subsequence, at least one sub-subsequence of \eqref{eq:wH1_cut} converges to 0.  Therefore we can assume that the following sequences converge in $\overline{\m R}$ or $\m S^{d-1}$:
\begin{gather} \label{seq_dicho}
t_n,  \quad \rho_n, \quad \omega_n , \quad \frac{\rho_n}{t_n}, \quad \frac{r_n}{t_n}, \quad \frac{r_n}{\rho_n}, \quad \frac{r_n - \rho_n}{t_n}, \quad \frac{r_n^2}{t_n} - t_n, \quad \frac{1}{t_n} \left( \frac{r_n^2}{\rho_n} - \rho_n \right), \quad \frac{r_n -t_n}{\rho_n}.
\end{gather}
Also observe the following claim

\begin{claim}
We can assume without loss of generality that one of the following four possibilities occur:
\begin{enumerate}
\item (whole space) $\m 1_{B(x_n, r_n)} \to \m 1$ a.e
\item (void) $\m 1_{B(x_n, r_n)} \to 0$ a.e.
\item (ball) There exists $x_\infty \in \m R^d$ and $r_\infty>0$ such that $\m 1_{B(x_n, r_n)} \to \m 1_{B(x_\infty, r_\infty)}$ a.e.
\item (half-space) There exist $\omega_\infty \in \m S^{d-1}$ and $c \in \m R$ such that $\m 1_{B(x_n, r_n)} \to \m 1_{x \cdot \omega_\infty \ge c}$ a.e.
\end{enumerate}
\end{claim}

For the claim: first assume that $\rho_n$ has as a finite limit. If $r_n \to +\infty$, we are in the case (whole space); if $r_n \to 0$, it is the (void) case; and if $r_n \to r_\infty>0$ has a finite positive limit, it is the case (ball). Now assume that $\rho_n \to +\infty$, and let $\omega_\infty$ be the limit of $\omega_n$. If $\rho_n -r_n$ tends to $-\infty$, we are in the (whole space) case; if $\rho_n - r_n \to +\infty$, it is the (void) case. Now if $\rho_n -r_n \to c \in \m R$ has a finite limit, we see that we are in the (half-space) senario.

\bigskip

We can now proceed with the proof of Lemma \ref{eq:wH1_cut} itself. If $t_n$ has a finite limit $t_\infty \in \m R$, then $\vec u(t_n)$ has a strong limit $\vec u(t_\infty)$ in $\dot H^1 \times L^2$. Therefore, by Lebesgue's dominated convergence theorem, we see that $ \m 1_{B(x_n, r_n)} \nabla_{t,x} u(t_n)$ has a strong limit $V\in (L^2)^{1+d}$, by inspecting each scenario of the claim. Moreover, the hypothesis of the Lemma is that $\nabla_{t,x}w_n \weak 0$ in $(L^2)^{1+d}$, therefore
\[  \int_{|x-x_n| \ge r_n} \overline{\nabla_{t,x} w_n} (x) \cdot \nabla_{t,x} u (t_n)(x) dx\to \int 0 \cdot V =0. \]

We now consider the case when $t_n$ has an infinite limit, and we can assume without loss of generality that $t_n \to +\infty$. In this case, our goal is to construct a solution $\vec v \in \q C(\m R, \dot H^1 \times L^2(\m R^d))$ to the linear wave equation \eqref{eq:lw} such that
\begin{gather} \label{eq:u_v}
\| \nabla_{t,x} v(t_n) - \m 1_{|x-x_n| \ge r_n} \nabla_{t,x} u(t_n) \|_{(L^2)^{1+d}} \to 0 \quad \text{as} \quad n \to +\infty.
\end{gather}
Assuming that such a $\vec v$ is constructed, the assumption on the weak convergence of $\vec w_n$ means that
\[ \int \nabla_{t,x} w_n(0,x) \cdot \nabla_{t,x} v(t_n) dx \to 0 \quad \text{as} \quad n \to +\infty, \]
from where we deduce \eqref{eq:wH1_cut} immediately. We therefore focus on the construction of such a $\vec v$.

We recall from \eqref{eq:u_h} that 
\begin{gather*} \nabla_{t,x} u(t,x) = \frac{1}{(2\pi |x|)^{\frac{d-1}{2}}}  h (|x|-t,x/|x|) \begin{pmatrix} 
-1 \\
x/|x|
\end{pmatrix}
+ \e(t,x),
\end{gather*}
where $\e(t,x) \to 0$ in $L^2(\m R^d,\m C^{d+1})$ and
\[ h(\rho,\omega) = \partial_\rho (e^{i \tau}f_{\omega}^-  + e^{-i\tau} g_{\omega}^+)(\rho). \]
The key point of the argument is the following:
\begin{claim} \label{cl:2}
 $\m 1_{|(\rho+t_n) \omega - x_n| \le r_n}$ has a limit for a.e $(\rho, \omega) \in \m R \times  \m S^{d-1}$, which we call $a(\rho,\omega)$.
\end{claim}

Of course, $a(\rho,\omega)$ is measurable and $0 \le a(\rho,\omega) \le 1$ a.e.-$(\rho, \omega) \in \m R \times \m S^{d-1}$. Let us posptone the proof of Claim \ref{cl:2} to the end, and assume it for now. Let us then define
\[ \tilde h(\rho, \omega) =  a(\rho,\omega) h(\rho,\omega). \]
The relevance of the definition comes from the 
\begin{claim} \label{cl3}
\begin{gather} \label{conv:h_th}  \frac{1}{(2\pi |x|)^{\frac{d-1}{2}}} \left( \m 1_{|x-x_n| \ge r_n} h(|x|-t_n,x/|x|) - \tilde h(|x|-t_n,x/|x|) \right) \to 0 \quad \text{in } L^2(\m R^d). 
\end{gather}
\end{claim}

\begin{proof}
As $ \tilde h(\rho, \omega) =  a(\rho,\omega) h(\rho,\omega)$, \eqref{conv:h_th} is equivalent to showing the convergence to 0, as $n \to +\infty$ of the quantity
\[ \int_{\m R} \int_{\m S^{d-1}}  |\m 1_{\rho \omega - x_n| \ge r_n}  - a(\rho-t_n,\omega) | | h(\rho -t_n,\omega)|^2 d\rho d\omega \]
or equivalently that
\begin{gather} \label{eq:cl3_key}
\int_{\m R} \int_{\m S^{d-1}}  |\m 1_{(\rho +t_n) \omega - x_n| \ge r_n}  - a(\rho,\omega) | | h(\rho,\omega)|^2 d\rho d\omega \to 0.
\end{gather}
Now, 
\[ \q F_{\rho \to \nu} h(\nu,\omega) = i \nu   \left( \m 1_{\nu \le 0} e^{i \tau}  \hat f (\nu \omega) +   \m 1_{\nu \ge 0} e^{-i \tau} \hat g (\nu \omega) \right), \]
so that by Parseval
\begin{align} 
\MoveEqLeft \int_{\m R} \int_{\m S^{d-1}} \left|  h(\rho,\omega) \right|^2  d\rho d\omega \\
& = (2 \pi)^{-1} \int_{\m S^{d-1}}\left( \int_{-\infty}^0 |\hat f(r \omega)|^2 |r|^{d+1} dr  +  \int_{0}^{+\infty} |\hat g(r \omega)|^2 r^{d+1} dr\right) d\omega \nonumber \\
& = (2 \pi)^{d-1} \left( \| f \|_{\dot H^1 (\m R^d)}^2 + \| g \|_{\dot H^1 (\m R^d)}^2\right) =  (2 \pi)^{d-1}2^{-1} \| (u_0,u_1) \|_{\dot H^1 \times L^2(\m R^d)}^2. \label{bd:h_omega}
\end{align}
Also, $|\m 1_{(\rho +t_n) \omega - x_n| \ge r_n}  - a(\rho,\omega) | \le 2$. Using Claim \ref{cl:2}, we see that Lebesgue's dominated convergence theorem applies and gives the convergence \eqref{eq:cl3_key}. Hence \eqref{conv:h_th} holds.
\end{proof}

It now suffices to define $\vec v$ from $\tilde h(\rho,\omega)$, which we do by following the steps, backwards, of getting $h(\rho,\omega)$ from $\vec u$. More precisely, define $\tilde f, \tilde g \in \dot H^1(\m R^d)$ by their Fourier transform
\begin{gather}\label{reconstructfg} \hat {\tilde f}(\xi) = -\frac{1}{i|\xi|} \frac{e^{-i \tau}}{|\xi|^{\frac{d-1}{2}}} \q F_{\m R} \tilde h(-|\xi|,-\xi/|\xi|), \quad \hat {\tilde g}(\xi) = \frac{1}{i |\xi|} \frac{e^{i\tau}}{|\xi|^{\frac{d-1}{2}}} \q F_{\m R} \tilde h(|\xi|, \xi/|\xi|).\end{gather}
Then it follows that for $(\nu, \omega) \in \m R \times \m S^{d-1}$
\[ \q F_{\m R} (\tilde h)(\nu,\omega) =  i \nu |\nu|^{\frac{d-1}{2}} \left( \m 1_{\nu \le 0} e^{i \tau}  \hat {\tilde f} (\nu \omega) +   \m 1_{\nu \ge 0} e^{-i \tau} \hat {\tilde g} (\nu \omega) \right), \]
so that for all $\omega \in \m S^{d-1}$, $\rho \in \m R$, and with the notation \eqref{def:f_omega}
\begin{gather} \label{eq:f_h_k}
\tilde h(\rho,\omega)\omega = \partial_\rho (e^{i \tau} \tilde f_\omega^- + e^{-i \tau} \tilde g_\omega^+)(\rho).
\end{gather}
Finally let 
\begin{gather}
\label{reconstructv01} v_0 = \tilde f + \tilde g, \quad v_1 = i|D| ( \tilde f- \tilde g ), 
\end{gather}
and denote $\vec v$ the solution to the linear wave equation \eqref{eq:lw} with data $(v_0,v_1)$. Arguing as for \eqref{eq:u_h}, we get
\begin{gather} \label{eq:v_th}
\nabla_{t,x} v(t,x) = \frac{1}{(2\pi |x|)^{\frac{d-1}{2}}} \begin{pmatrix}
-1 \\
x/|x|
\end{pmatrix} \tilde h(|x|-t,x/|x|) + \tilde \e(t,x)
\end{gather}
where $\tilde \e(t,x) \to 0$ in $L^2 (\m R^d, \m C^{d+1})$ as $t \to +\infty$.
Gathering together   \eqref{eq:v_th}, \eqref{conv:h_th} and \eqref{eq:u_h}, we see that \eqref{eq:u_v} holds: we are done, up to the proof of Claim \ref{cl:2}.
\end{proof}

\begin{proof}[Proof of Claim \ref{cl:2}]
We write $|(\rho+t_n) \omega - x_n| \le r_n \iff (\rho+t_n)^2 + \rho_n^2 - 2 \rho_{n}\omega_n \cdot \omega (\rho+t_n) \le r_n^2 \iff \rho \in [\rho_n^-(\omega), \rho_n^+(\omega)]$ where
\[ \rho_n^\pm(\omega) = -t_n + \rho_n \omega_n \cdot \omega \pm \sqrt{\rho_n^2 ((\omega_n \cdot \omega)^2-1) + r_n^2} \]
with the convention that $[\rho_n^-(\omega), \rho_n^+(\omega)]=\emptyset$ if $\rho_n^2 ((\omega_n \cdot \omega)^2-1) + r_n^2<0$.
Hence,
\[ \{ (\rho, \omega) : | (\rho+t_n) \omega - x_n| \le r_n \} = \bigcup_{\omega \in \m S^{d-1}} \left\{ (\rho, \omega) | \  \rho \in [\rho_n^-(\omega), \rho_n^+(\omega)] \right\}. \]
In terms of the rescaled variables $r'_n = r_n/t_n$ and $\rho_n'  = \rho_n/t_n$, this writes
\[ \rho_n^\pm(\omega) = t_n \left(-1 +  \rho_n' \omega_n \cdot \omega \pm \sqrt{{\rho_n'}^2 ((\omega_n \cdot \omega)^2-1) + {r_n'}^2} \right). \]
We claim that there exists a finite set $\q F$ (depending only on the limits of the sequences listed on \eqref{seq_dicho}) such that if $\omega \cdot \omega_\infty \notin \q F$, then $\rho_n^\pm(\omega)$ both have a limit in $\overline{\m R}$ as $n \to +\infty$.

Consider $a \in [-1,1]$ such that at least one of $t_n \left(-1 +  \rho_n' a \pm \sqrt{{\rho_n'}^2 (a^2-1) + {r_n'}^2} \right)$ does not have a limit in $\overline{\m R}$. As $t_n \to +\infty$ and using the fact that all terms have a limit in $\overline{\m R}$, this implies that
\[ -1 +  \rho_n' a + \sqrt{{\rho_n'}^2 (a^2-1) + {r_n'}^2} \to 0 \quad \text{or} \quad -1 +  \rho_n' a - \sqrt{{\rho_n'}^2 (a^2-1) + {r_n'}^2} \to 0. \]

We argue by disjunction of cases: denote $\alpha = \lim \rho_n'$, $\beta = \lim r_n'$, and $\ds \gamma =\lim \frac{r_n'}{\rho_n'}$.

$\bullet$ Assume first that $\alpha$ is finite and non zero, and $\beta$ also is finite. Then
\[ -1 +  \rho_n' a \pm \sqrt{{\rho_n'}^2 (a^2-1) + {r_n'}^2} \to -1 + \alpha a \pm \sqrt{\alpha (a^2-1) + \beta}. \]
Now by studying the variations of the functions $a \mapsto \alpha a \pm \sqrt{\alpha (a^2-1) + \beta}$, one concludes that there exists at most 2 points (for each function) where they take the value $1$; so $\q F$ is a  subset of these (at most 4) points.

$\bullet$ If $\alpha$ is finite and $\beta = +\infty$, $-1 +  \rho_n' a \pm \sqrt{{\rho_n'}^2 (a^2-1) + {r_n'}^2} \to \pm \infty$, so that $\q F$ is empty.

$\bullet$ If $\alpha =0$ and $\beta \ne 1$, then 
\[ -1 + \rho_n' a \pm \sqrt{{\rho_n'}^2 (a^2-1) + {r_n'}^2} \to -1 + \pm \beta \ne 0, \]
and therefore $\q F$ is empty.

$\bullet$ In the case $\alpha=0$ and $\beta =1$, clearly
\[ -1 +  \rho_n' a - \sqrt{{\rho_n'}^2 (a^2-1) + {r_n'}^2} \to -2,  \]
so that $\rho_n^-(\omega)$ always has a limit. For $\rho_n^+(\omega)$ we expand

\begin{align*}
 -1 +  \rho_n' a +  \sqrt{{\rho_n'}^2 (a^2-1) + {r_n'}^2} =  -1 + \rho_n' a + r_n' \left(1 + O(\rho_n'^2) \right)
\end{align*}
Denote $\delta = \lim \frac{r_n -t_n}{\rho_n} = \lim \frac{r_n'-1}{\rho_n'}$ and assume that $a \ne -\delta$, then
\[ \rho_n^+(\omega) = t_n \rho_n' \left( \frac{r_n'-1}{\rho_n'} + a + O(\rho_n') \right) \sim (\delta +a) \rho_n, \]
which has a limit in $\overline{\R}$, even if $\delta=+\infty$. Hence $\q F \subset \{ -\delta \}$.

$\bullet$ If $\alpha = +\infty$ and $\gamma \ne 1$, then $a \ne\pm \sqrt{a^2-1 + \gamma}$ so that 
\[  \rho_n' a \pm \sqrt{{\rho_n'}^2 (a^2-1) + {r_n'}^2} \sim \rho_n' (a \pm \sqrt{a^2-1 + \gamma }). \]
If $a \ne 0$, the limits exist (and are infinite), so that $\q F \subset \{ 0 \}$.

$\bullet$ In the case $\alpha=+\infty$ and $\gamma =1$, (which implies $\beta =+\infty$). Again we see that $\rho_n' a + \sqrt{{\rho_n'}^2 (a^2-1) + {r_n'}^2} \to +\infty$ so that it suffices to consider $\rho_n^-(\omega)$. Then for $a \ne 0$ we are allowed to expand
\begin{align*}
\rho_n' a - \sqrt{{\rho_n'}^2 (a^2-1) + {r_n'}^2} \sim \rho_n'  \frac{1}{2a} \left( 1-\frac{{r_n'}^2}{{\rho_n'}^2}  \right) \sim \frac{1}{2a t_n} \left(  \rho_n- \frac{r_n^2}{\rho_n} \right)
\end{align*}
This sequence has always a limit in $\overline{\m R}$, which is $1$ for at most one value of $a$. In that case, $\q F$ is made of this point and 0.

\bigbreak

We have exhausted all possibilities for the limits, and in all cases, $\q F$ is made of a finite number of points. If $\omega \cdot \omega_\infty \notin \q F$, we denote $\rho^\pm(\omega)$ the limits of $\rho^\pm_n(\omega)$ (whose existence were just shown above). 

Define
\[ \q N = (0,+\infty) \times \{ \omega \in \m S^{d-1} :  \omega \cdot \omega_\infty \in \q F \} \cup \bigcup_{\omega \in \m S^{d-1}} \{ (\rho^-(\omega), \omega), (\rho^+(\omega), \omega) \}. \]
Clearly $\q N$ is a negligeable subset of $(0,+\infty) \times \m S^{d-1}$. Also if $(\rho, \omega) \in (0,+\infty) \times \m S^{d-1} \setminus \q N$, by definition of a limit we see that either $|(\rho+t_n) \omega - x_n| \le r_n$ for all $n$ large enough, or $|(\rho+t_n) \omega - x_n| > r_n$ for all $n$ large enough; equivalently, $\m 1_{|(\rho+t_n) \omega - x_n| > r_n}$ has a limit as $n \to +\infty$.  \qedhere
\end{proof}

We can easily modify the proof of Lemma \ref{lem:wH1_conv_cut} to obtain a result in the setting of solutions to the half-wave equation (in the $L^2$ setting). More precisely, we have the following lemma, whose proof is left to the reader.

\begin{lem} \label{lem:wL2_conv_cut}
Let $f \in L^2(\m R^d)$, and $t_n \in \m R$, $x_n \in \m R^d$ and $r_n >0$ be three sequences. Assume that $w_n$ is a bounded sequence of $L^2(\m R^d)$ such that $e^{-i t_n |D|} w_n \weak 0$ and $e^{i t_n |D|} w_n \weak 0$ in $L^2(\m R^d)$. Then
\[ \int_{|x-x_n| \ge r_n} \overline{w_n} (x) (e^{i t_n |D|} f)(x)  dx \to 0 \quad \text{as} \quad n \to +\infty. \]
\end{lem}

We finally prove Proposition \ref{prop:ortho_cutoff}: it is similar to the proof of \cite[Corollary 8]{CoteKenigSchlagEquipart} to which we refer for further details. Expanding the norms we see that it suffices to prove that for $i \ne j$
\begin{gather*}
\int_{|x-x_n| \ge r_n} \overline{ \nabla_{t,x} U_L^i \left( - \frac{t_{i,n}}{\lambda_{i,n}}, \frac{x-x_{i,n}}{\lambda_{i,n}} \right) }
\cdot \nabla_{t,x} U_L^j \left( - \frac{t_{j,n}}{\lambda_{j,n}}, \frac{x-x_{j,n}}{\lambda_{j,n}}  \right) dx \to 0, \\
\int_{|x-x_n| \ge r_n}  \overline{\nabla_{x,t} w_{n}^J(0,x)}
\cdot \nabla_{t,x} U_L^j \left( - \frac{t_{j,n}}{\lambda_{j,n}}, \frac{x-x_{j,n}}{\lambda_{j,n}} \right) dx \to 0. 
\end{gather*}
Unscaling the integrals by $\lambda_{j,n}$ and then translating by $x_{j,n}$, we see that these expressions are of the form of \eqref{eq:wH1_cut}: the condition of weak convergence hold for the term in $U_L^i$ due to almost orthogonality of the profile, and for the term in $(w_{n,0}, w_{n,1})$ due to the construction of the profiles $U_L^j$ in terms of weak limit of rescaled and translated of $S(t_{j,n})(u_{n,0},u_{n,1})$.

\section{The operators $\Td$ and $\partial_s \Td$ on Sobolev spaces}
\label{sectRadsob}
\subsection{The Radon transform on the Schwartz class}

In this paragraph, we state the definitions and basic properties of the Radon transform on $\q S(\m R^d)$, for the convenience of the reader.  They are mostly classical: we refer to the paper \cite{Ludwig:66} or the reference book \cite{Helgason:book} for proofs and further details.

\bigskip

We recall the definition of the Radon transform $\Ra f$  of a function $f\in \mathcal{S}(\R^{d})$:
\begin{align*}
\forall (s,\omega) \in \m R \times \m S^{d-1}, \quad (\Ra f)(s,\omega):=\int_{\omega\cdot y=s}f(y)dy.
\end{align*}
 where $dy$ refers here to the surface measure on the hyperplane $\{y\in\R^{d};\omega\cdot y=s\}$. It can be checked that $\Ra f\in \mathcal{S}(\R\times \m S^{d-1})$ and $\Ra f$ is even in the sense that
 \begin{align}
 (\Ra f)(-\omega,-s)= (\Ra f)(\omega,s).
 \end{align}
 An important related operator is its adjoint $\q R^*$ defined for $\varphi\in \mathcal{S}(\R\times \m S^{d-1})$ by
 \begin{align}
\label{Radondualdef}
(\Ra^{*} \varphi)(x):=\int_{\omega\in \m S^{d-1}}\varphi(x\cdot \omega,\omega)d\omega,
\end{align}
so that $\q R^* \varphi \in \q S(\m R^d)$ and for $f\in \mathcal{S}(\R^{d})$ and $\varphi\in \mathcal{S}(\R\times \m S^{d-1})$,  the following duality relation holds:
\begin{align}
\label{dualRadon}
\int_{\R\times \m S^{d-1}} (\Ra f)(s,\omega)\overline{\varphi(s,\omega)}dsd\omega & = \int_{\R^{d}} f(x)\overline{(\Ra^{*} \varphi)(x)}dx.
\end{align}
Note that we have the important unitarity property in $L^{2}$ of the Radon transform (and in fact, in any $\dot H^s$), up to a constant related to $c_0$ (which appeared in \eqref{def:tau}).

\begin{prop}[Unitarity]
\label{lmRadonunit}
For every $f\in \mathcal{S}(\R^{d})$, we have
\begin{align}
\label{unitRadon}
\int_{\R_y}| f(y)|^2dy & = c_{0}^{2}\int_{\R_{s}\times \m S^{d-1}_{\omega}} |\partial_{s}^{\frac{d-1}{2}}\Ra f(\omega,s)|^2ds~d\omega, \\
\label{unitRadonH1}\int_{\R_y}|\nabla f(y)|^2dy & = c_{0}^{2}\int_{\R_{s}\times \m S^{d-1}_{\omega}} |\partial_{s}^{\frac{d+1}{2}}\Ra f(\omega,s)|^2ds~d\omega, \\
\label{actionDRadon} \Ra (\partial_{x_j}f) & = \omega_{j} \partial_{s}\Ra f,\\
\label{actionDeltaRadon}\Ra(\Delta f) & = \partial_{s}^{2}\Ra f.
\end{align}
\end{prop}

\begin{proof}
We refer to {\cite[Theorem 3.13 page 31]{Laxbook}} where the proof is done in odd dimension, but holds for even dimension as well.  \eqref{unitRadonH1} is obtained by combining \eqref{unitRadon} and \eqref{actionDRadon}. Note also that Lemma \ref{lem:T_iso} can actually provide a proof of this Lemma \ref{lmRadonunit} once the link between $\Ra$ and $\Td$ is precised, as will be done in Lemma \ref{lmRadon} below. We also refer to Lemma 2.1 of \cite{Helgason:book}
\end{proof}

\begin{prop}[Inverse]
For every $f\in \mathcal{S}(\R^{d})$, we have
\begin{align*}
f=\Ra^{*} (c_0 |D_s|^{d-1}) \Ra f.
\end{align*}
\end{prop}

Note that in odd dimension, $(c_0 |D_s|^{d-1})$ is a differential operator.

The extension of the Radon transform to distributions presents some difficulties mainly coming from the fact that $\Ra^{*}$ does not obviously preserve decay: for example, it does not map $\mathcal{S}(\R^{d})$ or $ \mathcal{\q D}(\R^{d})$ into itself. 

One can however easily extend $\Ra$ to compactly supported distributions $\mathcal{E}'(\R^{d})$.

\begin{prop}[Radon transform on $\q E'$]
 $\Ra^{*}$ maps (continuously) $\mathcal{E}(\R\times \m S^{d-1})= \q C^{\infty}(\R\times \m S^{d-1})$ into $\mathcal{E}(\R^{d})= \q C^{\infty}(\R^{d})$. 
 As a consequence, for $u \in \mathcal{E}'(\R^{d})$, one defines its Radon transform $\Ra u \in \mathcal{E}'(\R\times \m S^{d-1})$ by the formula
\begin{align}
\label{RaDistrib}
\left< \Ra u,\varphi \right>_{\mathcal{E}'(\R\times \m S^{d-1}),\mathcal{E}(\R\times \m S^{d-1})}
:=\left<  u,\Ra^{*}\varphi \right>_{\mathcal{E}'(\R^{d}),\mathcal{E}(\R^{d})}.
\end{align}
Furthermore,
\begin{align}
\label{actionDeltaRadondual}
\forall \varphi\in \mathcal{S}(\R\times \m S^{d-1}), \quad \Delta \Ra^{*} \varphi & = \Ra^{*}(\partial_{s}^{2}\varphi),
\end{align}
so that  \eqref{actionDeltaRadon} also hold in $\mathcal{E'}(\R^{d})$:
\begin{align}
\label{actionDeltaRadonD} 
\forall u\in \mathcal{E'}(\R^{d}), \quad \Ra(\Delta u) & = \partial_{s}^{2}\Ra u. 
\end{align}
\end{prop}

\begin{proof}
See \cite[Section 4]{Ludwig:66}, and in particular Theorem 4.9.
\end{proof}

An important result for our purpose is the description of the range of the Radon transform. We will make extensive use of the following result, which describes the images of Schwartz class functions.

\begin{thm}\cite[Theorem 2.1]{Ludwig:66}
\label{thmLudwig}
$g\in \mathcal{S}(\m S^{d-1}\times \R)$ can be written $g=\Ra f$ for some $f\in \mathcal{S}(\R^{d})$ if and only if the following two conditions are fulfilled
\begin{enumerate}
\item \label{condi1} $g$ is even, that is $g(\w,s)=g(-\w,-s)$ for any $(\w,s)\in \m S^{d-1}\times \R$.
\item if $Y_{\ell}(\w)$ is a spherical harmonics of degree $l$ and if $0\le k<l$ integers, then
\begin{align}
\label{condimoment}
\int_{\w\in \m S^{d-1}}\int_{s\in\R}g(s,\w)s^{k}Y_{\ell}(\w)dsd\w=0.
\end{align}
\end{enumerate}
\end{thm}

In proving Theorem \ref{thmexteriorgene}, it will be important for us to know the Radon transform of product of radial functions and spherical harmonics. This makes use of the Gegenbauer polynomials, whose definition and properties are recalled below.

\begin{definition}
\label{defGegen}
Let $\lambda\in \R^{*}$. For any $l\in \N$, we define the Gegenbauer polynomials $C_{l}^{\lambda}$ by iteration: $C_{0}^{\lambda}(t)=1$, $C_{1}^{\lambda}(t)=2\lambda t$ and for $l \ge 2$,
\begin{align*}
\quad C_{l}^{\lambda}(t)=\frac{1}{l}\left(2t(l+\lambda-1)C_{l-1}^{\lambda}(t)-(l+2\lambda-2)C_{l-2}^{\lambda}(t)\right).
\end{align*}
For $\lambda=0$, the Gegenbauer polynomials are the Chebychev polynomials where similar formula holds, but will not be used here since it corresponds to the even dimension $d=2$.
\end{definition}

\begin{prop}[{\cite[Section 22, page 773]{AS64}}]
\label{propGegen}
$C_{l}^{\lambda}$ is a polynomial of degree $l$ which is an even (resp. odd) function if $l$ is even (resp. odd).

The polynomials $t\mapsto C_{l}^{\lambda}(t)$ are orthogonal on the interval $[-1,1]$ for the weight function $(1-t^{2})^{\lambda-\frac{1}{2}}$. More precisely, for $l$, $m\in \N$, $\lambda>-1/2$, we have
\begin{align*}
\int_{-1}^{1}C_{l}^{\lambda}(t)C_{m}^{\lambda}(t)(1-t^{2})^{\lambda-\frac{1}{2}}dt= \delta_{l,m} \frac{\pi2^{1-2\lambda}\Gamma(l+2\lambda)}{l! (l+\lambda)\Gamma(\lambda)^{2}}.
\end{align*}
\end{prop}

We will be using mostly in the specific case $\lambda=\frac{d}{2}-1$, because those polynomial appear when computing the Radon transform of functions involving spherical harmonics.

\begin{prop}
\label{lmRadldual}
1. If $h(s,\omega)=g(s)Y_{\ell}\left(\omega\right)$, where $g\in \q C^{\infty}(\R)$ and $g(-s)=(-1)^{l}g(s)$, then $(\Ra^{*} h)(x)=W(|x|)Y_{\ell}\left(\frac{x}{|x|}\right)$ where
\begin{align}
\label{formuleGegendual}
W(r)=(\Ra_{l}^{*}w)(s) :=\frac{|\m S^{d-1} |}{C_{l}^{d/2-1}(1)}\int_{-1}^{1} C_{l}^{d/2-1}(t)g(rt)\left(1-t^{2}\right)^{\frac{d-3}{2}}dt.
\end{align}
We have $W(r)=r^{l}v(r^{2})$ where $v\in \q C^{\infty}(\R)$. 

\label{lmRadl}
2. Similarly, if $f(x)=w(|x|)Y_{\ell}\left(\frac{x}{|x|}\right)$, where $w(r)=r^{l}v(r^{2})$ and $v\in \mathcal{S} (\m R)$, then $(\Ra f)(s,\omega)=k(s)Y_{\ell}(\omega)$ where
\begin{align}
\label{formuleGegen}
k(s)=(\Ra_{l}w)(s) :=\frac{|\m S^{d-1} |}{C_{l}^{d/2-1}(1)}\int_{|s|}^{+\infty} C_{l}^{d/2-1}\left(\frac{s}{r}\right)w(r)r^{d-2}\left(1-\frac{s^{2}}{r^{2}}\right)^{\frac{d-3}{2}}dr.
\end{align}
The function $k$ is in $\mathcal{S} (\m R)$ and satisfies $k(-s)=(-1)^{l}k(s)$. 

3. Furthermore, under the notations of 2., $w=\Ra_{l}^{*} (c_0 |D_s|^{d-1}) \Ra_{l} w$ and moreover
\begin{align*}
\int_{0}^{+\infty}|w(r)|^{2}r^{d-1}d = c_0 \int_{-\infty}^{+\infty} \left| |D_s|^{\frac{d-1}{2}}\Ra_{l}w\right|^{2}ds.
\end{align*}
\end{prop}

\begin{proof}
These results are the content of Lemmata 5.1 and 5.2, and of Theorem 5.1 of \cite{Ludwig:66} 
\end{proof}

\subsection{{Extension for the operators $\Td$ and $\partial_s \Td$}}


Let us first relate the operator $\Td$ and the Radon transform, and, in particular, proves relation \eqref{def:T_odd}. We start with expressing the Radon transform via a partial Fourier transform.

\begin{lem}
\label{lmRadon}
The map defined on $\q S(\m R^d)$ by $\mathcal{F}_{\nu \to s}^{-1}\left[\hat f(\nu\omega)\right]$ is the Radon transform in the direction $\omega$. That is 
\begin{align} \label{eq:R_F}
\Ra f(s,\omega):=\int_{\omega\cdot y=s}f(y)dy=\mathcal{F}_{\nu \to s}^{-1}\left[\hat f(\nu\omega)\right](s).
\end{align}
As a consequence, one has the equality as operators $\q S(\m R^d) \to \q S'(\m R \times \m S^{d-1})$:
\[ \q T = m_{d}(D_{s})\Ra \quad \text{where} \quad m_{d}(\nu) :=c_{0} |\nu|^{\frac{d-1}{2}}\left(e^{i\tau }\m 1_{ \nu<0} +e^{-i\tau }\m 1_{ \nu \geq0} \right). \]
In particular, if the space dimension $d$ is odd, then \eqref{def:T_odd} holds:
\[ \q T = c_0 (-1)^{\frac{d-1}{2}} \partial_s^{\frac{d-1}{2}} \Ra. \]
\end{lem}

\begin{proof}
For fixed $\omega$, we compute for $f\in \mathcal{S}(\R^{d})$ the 1D Fourier transform of $s \mapsto \Ra f(s,\omega)$. Then
\begin{align*}
(\q F_{\m R} \Ra f(\cdot,\omega))(\nu) & = \int  \Ra f(s,\omega) e^{-i \nu s} ds = \int_s \int_{\omega\cdot y=s} f(y) e^{-i \nu s} dy ds \\
& =\int_s \int_{y\in s \omega+ \omega^\perp } f(y) e^{-i \nu \omega \cdot y} dy ds \\
& = \int_{x \in \m R^d} f(x) e^{-i \nu \omega \cdot x} dx = \hat f(\nu \omega).
\end{align*}
(We used Fubini's theorem with $\m R^d = \bigcup_{s \in \m R} (s \omega+ \omega^\perp)$ for the last line). This proves \eqref{eq:R_F}.
The second equality is then direct via 1D Fourier transform, using the definition \eqref{def:T}.

If furthermore $d$ is odd, recalling that $\tau =  \frac{d-1}{4} \pi$, we distinguish the odd cases modulo 4:
\begin{itemize}
\item if $d=4k+1$, $m_{d}(\nu)=(-1)^{k}c_{0}|\nu|^{2k}$, so that $\Td = c_0 \partial_s^{2k} \q R$,
\item if $d=4k+3$, $m_{d}(\nu)=i c_0 |\nu|^{2k+1}(-1)^{k+1}(-\sgn(\nu))=i(-1)^{k+1} c_0\nu^{2k+1}$, so that $\Td = -c_0 \partial_s^{2k+1} \q R$.
\end{itemize}
This yield the last equality, that is \eqref{def:T_odd}.
\end{proof}

In order to apply homogeneity arguments, we would like to extend the previous applications to other spaces, in particular, containing homogeneous function of the form $|x|^{\alpha}Y_{\ell}\left(\frac{x}{|x|}\right)$ with $-d<\alpha<1-d/2$ that are not in $\dot{H}^{1}$ because of the behaviour close to zero. The purpose of this section is then to properly define $\partial_{s}\Td$ and $\Td$ in some larger distributional sense.

We will present the statement in a context adapted to $\dot H^1$ and $L^2$, as it is our interest here. One could proceed via partial Fourier transform, specially in view of Lemma \ref{lem:T_iso}. However, we crucially relie on locality properties in our argument, which follow from that of the Radon transform. In order to achieve this, it is natural to proceed by duality, and for this, the restriction to odd dimension appears naturally. We begin with definition for the adjoint of $\Td$.

\begin{definition} \label{defTzhdual}
We define a map $\Td^{*}: \mathcal{S}(\R\times \m S^{d-1}) \to \mathcal E(\m R^d)$, by letting for $\varphi\in \mathcal{S}(\R\times \m S^{d-1})$,
\begin{align*}
\Td^{*}\varphi & :=  c_{0}\Ra^{*} \partial_{s}^{\frac{d-1}{2}} \varphi.
\end{align*}
\end{definition}

We also need $L^2$ type spaces with symmetry.
\begin{definition}
We denote:
\begin{align*}
L^{2}_{odd}(\m R \times \m S^{d-1} ) & :=\left\{g\in L^{2}(\m R \times \m S^{d-1}); \ g(s,\w,s)=-g(-s,-\w), a.e. \right\}, \\
L^{2}_{even}(\m R \times \m S^{d-1}) & :=\left\{g\in L^{2}(\m R \times \m S^{d-1}); \ g(s,\w)=g(-s,-\w), a.e. \right\}.
\end{align*}
\end{definition}

\begin{prop}
\label{propdsTH1} 
Assume $d\ge 3$ is odd. Then, the operator $\Td$ can be extended
\[ \Td:  \bigcup_{\rho \in \R}H^{\rho}(\R^{d})  \to \mathcal{S}'(\R\times \m S^{d-1}) \]
so that it satisfies
\begin{enumerate}
\item if $f\in \mathcal{S}(\R^{d})$, $\Td f=c_{0}(-1)^{\frac{d-1}{2}} \partial_{s}^{\frac{d-1}{2}}\Ra f$,
\item if $u\in H^{\rho}(\R^{d})$ for some $\rho\in \R$ and $\varphi\in \mathcal{S}(\R\times \m S^{d-1}))$,  
\[ \left<\Td u,\varphi \right>_{\mathcal{S}'(\R\times \m S^{d-1}),\mathcal{S}(\R\times \m S^{d-1}))} =  c_{0}\left<  u,\Ra^{*} \partial_{s}^{\frac{d-1}{2}}\varphi \right>_{H^{\rho}(\R^{d}),H^{-\rho}(\R^{d})}. \]

\item \label{itemactionDeltaRadonDhs}   if $u\in H^{\rho}(\R^{d})$ then
\begin{align} \label{actionDeltaRadonDhs} 
\Td(\Delta u) & = \partial_{s}^{2}\Td u.
\end{align}
\item \label{supporprop}if $\ds u,~v\in \bigcup_{\rho \in \R}H^{\rho}(\R^{d})$ satisfy $u_{|\{|x|>R\}}=v_{|\{|x|>R\}}$ in the distributional sense, then 
$(\Td u)_{|\{|s|> R\}}=(\Td v)_{|\{|s|> R\}}$. (Here and below, $\{ |s| > R \} = \{ (s,\omega) \in \m R \times \m S^{d-1} ; |s| >R \}$ takes into account both variables $(s,\omega)$).
\end{enumerate}
Furthermore,  $\partial_s \Td$ extends to a linear map $\dot H^1 \to \mathcal{S}'(\R\times \m S^{d-1})$ as follows: if $u\in \dot{H}^{1}(\R^{d})$ and $\varphi\in \mathcal{S}(\R\times \m S^{d-1}))$, 
\begin{align} \label{formdualH1}\left<\partial_{s} \Td u,\varphi \right>_{\mathcal{S}'(\R\times \m S^{d-1}),\mathcal{S}(\R\times \m S^{d-1})}=  c_{0} \left< \nabla u,\nabla \Ra^{*} \partial_{s}^{\frac{d-3}{2}}\varphi \right>_{L^{2}(\R^{d})^{d},L^{2}(\R^{d})^{d}}.
\end{align}
Therefore, $\partial_s \Td$ is defined on $\bigcup_{\rho \in \R}H^{\rho}(\R^{d})+\dot{H}^{1}(\R^{d})$, and enjoy the locality property:
if $\ds u,~v\in \bigcup_{\rho \in \R}H^{\rho}(\R^{d})+\dot{H}^{1}(\R^{d})$ satisfy 
\[ u_{|\{|x|>R\}}=v_{|\{|x|>R\}} \]
in the distributional sense, then 
\[ (\partial_{s}\Td u)_{|\{|s|> R\}}=(\partial_{s}\Td v)_{|\{|s|> R\}}. \]
\end{prop}

As mentioned above, the proof of Proposition \ref{propdsTH1} is essentially done by duality. The starting point for extending $\Td$ is the next property.

\begin{lem}
\label{lmTd*Hs}Assume $d$ is odd.

For any $\rho \in \R$, $\Td^{*}$ maps continuously $\mathcal{S}(\R\times \m S^{d-1})$ into $H^{\rho}(\R^{d})$. 

Moreover, $\nabla \Ra^{*} \partial_{s}^{\frac{d-3}{2}}$ maps (continuously) $L^{2}(\R\times \m S^{d-1}) \to L^{2}(\R^{d})^{d}$, and $\mathcal{S}(\R\times \m S^{d-1}) \to H^{\rho}(\R^{d})^{d}$ for all $\rho \in \m R$. 
\end{lem}

\begin{proof} It is enough to prove the results for all $\rho \in \m N$.

Regarding $\Td^{*}$, the case $\rho=0$ is given by duality from \eqref{dualRadon} and \eqref{unitRadon}. Let $k \in \m N$. For $\varphi\in \mathcal{S}(\R\times \m S^{d-1})$, in view of \eqref{actionDeltaRadondual}, we see that $\Delta^{k}\Td^{*}\varphi\in L^{2}(\R^{d})$ for any $k\in \N$ so that $\Td^{*}$ maps $\mathcal{S}(\R\times \m S^{d-1})$ into $H^{2k}(\R^{d})$. It gives the result.

For $\nabla \Ra^{*} \partial_{s}^{\frac{d-3}{2}}$, we also work by duality, first for $\rho=0$. \eqref{dualRadon} gives for $j\in \{1,...,d\}$, after several integration by parts 
\begin{align*}
\int_{\R\times \m S^{d-1}} \partial_{s}^{\frac{d-3}{2}}(\Ra \partial_{x_{j}}f)(s,\omega)\overline{\varphi(s,\omega)}dsd\omega & = (-1)^{\frac{d-1}{2}} \int_{\R^{d}} f(x)\partial_{x_{j}}\overline{(\Ra^{*} \partial_{s}^{\frac{d-3}{2}}\varphi)(x)}dx.
\end{align*}
Using \eqref{actionDRadon} for the left hand side, this gives
\begin{align*}
\int_{\R^{d}} f(x) \partial_{x_{j}}\overline{(\Ra^{*} \partial_{s}^{\frac{d-3}{2}}\varphi)(x)}dx =  (-1)^{\frac{d-1}{2}}\int_{\R\times \m S^{d-1}} \omega_{j}\partial_{s}^{\frac{d-1}{2}}(\Ra f)(s,\omega)\overline{\varphi(s,\omega)}dsd\omega
\end{align*}
Due to \eqref{unitRadon} and Cauchy-Schwarz inequality, we conclude
\begin{align*}
\int_{\R^{d}} f(x) \partial_{x_{j}}\overline{(\Ra^{*}\partial_{s}^{\frac{d-3}{2}}\varphi)(x)}dx\le c_{0}\nor{\varphi}{L^{2}(\R\times \m S^{d-1})}\nor{f}{L^{2}(\R^{d})}.
\end{align*}
This implies $\nor{\partial_{x_{j}}\Ra^{*} D_{s}^{\frac{d-3}{2}}\varphi}{L^{2}(\R^{d})}\le c_{0}\nor{\varphi}{L^{2}(\R\times \m S^{d-1})}$ by duality, which is the case $\rho=0$. As for the other term, we treat the other regularities by applying $\Delta $ which gives the same result by changing $\varphi$ to $\partial_{s}^{2}\varphi$.
\end{proof}

\begin{proof}[Proof of Proposition \ref{propdsTH1}]
\eqref{dualRadon} gives directly that for every $f\in \mathcal{S}(\R^{d})$ and $\varphi\in \mathcal{S}(\R\times \m S^{d-1})$
\begin{align}
\label{dualTd}
\int_{\R\times \m S^{d-1}} (\Td f)(s,\omega)\overline{\varphi(s,\omega)}dsd\omega & = \int_{\R^{d}} f(x)\overline{(\Td^{*} \varphi)(x)}dx.
\end{align}

Therefore, as a by-product of Lemma \ref{lmTd*Hs}, we can extend $\Td$ as an operator on $H^\rho$ and $\partial_{s}\Td$ on $\dot H^1$. Indeed, for any $u\in H^{\rho}(\R^{d})$, with $\rho\in \R$, we can define a distribution $\Td u\in \mathcal{S}'(\R\times \m S^{d-1})$ by 
\begin{align}
\label{RaDistribHs}
\left<\Td u,\varphi \right>_{\mathcal{S}'(\R\times \m S^{d-1}),\mathcal{S}(\R\times \m S^{d-1}))}:=\left<  u,\Td^{*}\varphi \right>_{H^{\rho}(\R^{d}),H^{-\rho}(\R^{d})}.
\end{align}
The bracket defines a distribution due to Lemma \ref{lmTd*Hs}. \eqref{dualTd} gives that \eqref{RaDistribHs} coincides with \eqref{def:T_odd} for $u\in\mathcal{S}(\R^{d})$.

With this new definition, by duality, we still have the formula 
\begin{align*}
\forall u\in H^{\rho}(\R^{d}), \quad \Td(\Delta u) & = \partial_{s}^{2}\Td u,
\end{align*}
where $\Delta$ has to be understood as an operator $H^{\rho}(\R^{d}) \to H^{\rho-2}(\R^{d})$ (both being subspaces of $\mathcal{S}'(\R^{d})$), while $\partial_{s}^{2}$ is understood acting on $\mathcal{S}'(\R\times \m S^{d-1})$. That means the equality makes sense in $\mathcal{S}'(\R\times \m S^{d-1})$.

Similarly, for $u\in \dot{H}^{1}$, formula \eqref{formdualH1} and Lemma \ref{lmTd*Hs} allow to define $\partial_{s}\Td\in \mathcal{S}'$.
\bigskip

We have extended $\partial_s \Td$ in two ways, on $ \cup_{\rho \in \R}H^{\rho}(\R^{d})$ and on $\dot{H}^{1}(\R^{d})$. In order to see that it indeed defines an extension $\cup_{\rho \in \R}H^{\rho}(\R^{d})+\dot{H}^{1}(\R^{d}) \to \mathcal{S}'(\R\times \m S^{d-1})$, it only remains to check that for $u\in  \cup_{\rho \in \R}H^{\rho}(\R^{d})\cap\dot{H}^{1}(\R^{d})$, the two definitions coincide. That relies on verifying that for $u\in  \cup_{\rho \in \R}H^{\rho}(\R^{d})\cap\dot{H}^{1}(\R^{d})$ and $\varphi\in \mathcal{S}(\R\times \m S^{d-1}))$, we have
\[ - c_{0}\left<  u,\Ra^{*} \partial_{s}^{\frac{d-1}{2}+1}\varphi \right>_{H^{\rho}(\R^{d}),H^{-\rho}(\R^{d})}=c_{0} \left< \nabla u,\nabla \Ra^{*} D_{s}^{\frac{d-3}{2}}\varphi \right>_{L^{2}(\R^{d})^{d},L^{2}(\R^{d})^{d}}, \]
which is easily done using again \eqref{actionDeltaRadondual}.

\bigskip

We now prove the support properties \eqref{supporprop}. A similar result is contained in \cite[Theorem 4.9]{Ludwig:66} for $\Ra$, and follows from duality; we give a proof in the case of $\Td$ for completeness. Note that for this point, we use very strongly that the dimension is odd and that the operators are local.

We first notice that $u-v\in \mathcal{E}'(\R^{d})\subset \cup_{\rho\in \R}H^{\rho}(\R^{d})$ is supported in $B(0,R)$. Let $\varphi\in C^{\infty}_{0}(\R\times \m S^{d-1})$ so that $\varphi$ is supported in $\{|s|> R\}\times \m S^{d-1}$, that is $\varphi(s,\omega)=0$ for $|s|\le R$ and in particular, $\partial_{s}^{\frac{d-1}{2}}\varphi =0$ for $|s|\le R$. By the definition  \eqref{Radondualdef} of $\Ra^{*}$, it is clear that it implies $(\Ra^{*}\partial_{s}^{\frac{d-1}{2}}\varphi)(x)=0$ for $|x|\le R$. Now, we can compute
\begin{align*}
\left< \Td u - \Td v,\varphi \right>_{\mathcal{S}'(\R\times \m S^{d-1}),\mathcal{S}(\R\times \m S^{d-1}))} & = \left<  u-v,\Td^{*} \varphi \right>_{H^{\rho}(\R^{d}),H^{-\rho}(\R^{d})}\\
 & = c_{0}\left<  u-v,\Ra^{*} \partial_{s}^{\frac{d-1}{2}} \varphi \right>_{H^{\rho}(\R^{d}),H^{-\rho}(\R^{d})}.
\end{align*}
This is zero thanks to the respective support properties of $u-v$ and $(\Ra^{*}\partial_{s}^{\frac{d-1}{2}}\varphi)$.  A very similar computation yields the result for $\partial_s \Td$.
\end{proof}
We define the scaling operator on $\R^{d}$, (resp. $\R\times \m S^{d-1}$), for $\lambda>0$ by 
\begin{align*}
M_{\lambda}(x) & = \lambda x, \quad x\in \R^{d}\\
M_{\lambda}(s,\omega) & = (\lambda s,\omega), \quad (s,\omega)\in \R\times \m S^{d-1}.
\end{align*}
(with a slight abuse of notation), so that for $u\in \mathcal{S}'(\R^{d})$, $f\in \mathcal{S}(\R^{d})$ and $v\in  \mathcal{S}'(\R\times \m S^{d-1})$, $\varphi\in \mathcal{S}(\R\times \m S^{d-1})$, we have
\begin{align*}
\left< u\circ M_{\lambda},f \right>_{\mathcal{S}'(\R^{d}),\mathcal{S}(\R^{d})}  & = \lambda^{ -d}\left<u, f \circ M_{1/\lambda}\right>_{\mathcal{S}'(\R^{d}),\mathcal{S}(\R^{d})}\\
\left<v \circ M_{\lambda},\varphi \right>_{\mathcal{S}'(\R\times \m S^{d-1}),\mathcal{S}(\R\times \m S^{d-1}))}
 & = \lambda^{ -1}\left<v, \varphi\circ M_{1/\lambda}\right>_{\mathcal{S}'(\R\times \m S^{d-1}),\mathcal{S}(\R\times \m S^{d-1}))}
\end{align*}
We say that a distribution $u\in \mathcal{S}'(\R^{d})$ is homogeneous of order $\alpha$ if $u\circ M_{\lambda}=\lambda^{\alpha}u$.

\begin{lem}
\label{lmRahomogen}Assume $d$ is odd.

If $u\in H^{\rho}(\R^{d})$, then $(\partial_{s}\Td) (u\circ  M_{\lambda})= \lambda^{1- \frac{d-1}{2}}(\partial_{s}\Td u) \circ M_{\lambda}$.

In particular, if $u$ is homogeneous of order $\alpha$, then $\Td u$ is homogeneous of order $\alpha+\frac{d-1}{2}$ and $\partial_{s}\Td u$ is homogeneous of order $\alpha-1+\frac{d-1}{2}$ (the latter is also true if $u \in \dot H^1(\m R^d)$).
\end{lem}
\begin{proof}
We easily get that for $\varphi\in \mathcal{S}(\R\times \m S^{d-1})$,  $\Ra^{*}  (\varphi \circ M_{1/\lambda})(x)=(\Ra^{*} \varphi)(x/\lambda)$, so that $\Td^{*}  (\varphi \circ M_{1/\lambda})=c_{0}\Ra^{*} \partial_{s}^{\frac{d-1}{2}}  (\varphi \circ M_{1/\lambda})=\lambda^{-\frac{d-1}{2}}(\Td^{*} \varphi)\circ  M_{1/\lambda}$. So, we compute
\begin{align*}
\MoveEqLeft \left<(\partial_{s}\Td u) \circ M_{\lambda},\varphi \right>_{\mathcal{S}'(\R\times \m S^{d-1}),\mathcal{S}(\R\times \m S^{d-1}))} \\
& = -\lambda^{ -1}\left< u,(\Td^{*}\partial_{s})  (\varphi \circ M_{1/\lambda}) \right>_{H^{s}(\R^{d}),H^{-s}(\R^{d})}\\
& = -\lambda^{ -2-\frac{d-1}{2}}\left< u,(\Td^{*}\partial_{s} \varphi)\circ  M_{1/\lambda} \right>_{H^{s}(\R^{d}),H^{-s}(\R^{d})} \\
& = -\lambda^{-2-\frac{d-1}{2}+d}\left< u\circ  M_{\lambda},(\Td^{*} \partial_{s}\varphi) \right>_{H^{s}(\R^{d}),H^{-s}(\R^{d})}\\
& = \lambda^{-1+ \frac{d-1}{2}}\left< (\partial_{s}\Td) (u\circ  M_{\lambda}), \varphi) \right>_{\mathcal{S}'(\R\times \m S^{d-1}),\mathcal{S}(\R\times \m S^{d-1}))}. \qedhere
\end{align*}
\end{proof}

We saw in Lemma \ref{lem:T_iso} that $\Td$ was, in some sense, isometric on $L^2$ (and $\partial_s \Td$ on $\dot H^1)$. Below, we precise the range.

\begin{lem} \label{lem:range}
We consider here the restriction of $\Td$ to $L^{2}(\R^{d})$ (which we still denote $\Td$). We saw that $\Td$ is an isometry from $L^{2}(\R^{d}) \to L^{2}(\R \times \m S^{d-1})$. Then
\[ Range(\Td) = \begin{cases}
L^{2}_{even}(\R \times \m S^{d-1}) & \text{ if } d \equiv 1 [4] \\
L^{2}_{odd}(\R \times \m S^{d-1})& \text{ if } d \equiv 3 [4]
\end{cases}. \]
Similarly, the restriction $\partial_s \Td: \dot H^1 (\R^{d}) \to L^{2}(\m S^{d-1}\times \R)$ is isometric and
\[ Range(\partial_s \Td) = \begin{cases}
L^{2}_{odd}(\R \times \m S^{d-1}) & \text{ if } d \equiv 1 [4] \\
L^{2}_{even}(\R \times \m S^{d-1})& \text{ if } d \equiv 3 [4]
\end{cases}. \]
\end{lem}

\begin{proof}
The extension and unitarity comes from \eqref{unitRadon}. Concerning the range of $\Td$, we assume that $d=4k+1$ (that is $d \equiv 1 [4]$) to fix ideas, and it is enough to prove that $L^{2}_{even}(\m S^{d-1}\times \R) \subset \overline{Range(\Td)}$: indeed, $Range(\Td)$ is closed since $\Td$ is an isometry,  and is clearly contained in $L^{2}_{even}(\m S^{d-1}\times \R)$.  

In view of Theorem \ref{thmLudwig}, it suffices to prove that any function in $h\in L^{2}_{even}(\m S^{d-1}\times \R)$ can be approximated by a function of the form $\partial_{s}^{2k}g$ where $g\in \mathcal{S}(\m S^{d-1}\times \R)$ is even and satisfies \eqref{condimoment}. Decompose $h(\w,s)=\sum_{\ell \in \m M} h_{\ell} (s)Y_{\ell}(\w)$ (recall that $(Y_{\ell})_{\ell \in \m M}$ form an orthonormal basis of spherical harmonics of $L^{2}(\m S^{d-1})$ of degree $l = l(\ell)$).  The condition that $h$ be even can be written $h_{\ell}(-s)=(-1)^{l}h_{\ell}(s)$. Given $\e>0$, we are looking for $g_{\ell}\in \mathcal{S}(\R)$ so that 
\begin{itemize}
\item $\sum_{k,l} \nor{h_{k,l}-|D_{s}|^{\frac{d-1}{2}}g_{k,l}}{L^{2}(\R)}^{2}\le \e$,
\item $g_{k,l}(-s)=(-1)^{l}g_{k,l}(s)$,
\item $\ds \int_{s\in\R} g_{k,l}s^{j}ds =0$ for $j=0,\cdots, l-1$.
\end{itemize}
Translating this conditions in the Fourier side, yields
\begin{itemize}
\item $\sum_{k,l} \nor{\widehat{h}_{k,l}-|\xi_{s}|^{\frac{d-1}{2}}\widehat{g}_{k,l}}{L^{2}(\R_{\xi})}^{2}\le \e$,
\item  $\widehat{g}_{k,l}(-\xi_{s})=(-1)^{l}\widehat{g}_{k,l}(\xi_{s})$
\item $\ds \frac{d^{j}}{d\xi^{j}}\widehat{g}_{k,l}(0) =0$ for $j=0,\cdots, l-1$.
\end{itemize}
These conditions can clearly be met:  this gives the result for $Range(\Td)$. 

One can argue in a similar way in dimension $d \equiv 3 [4]$, and for $\partial_s \Td$.
\end{proof}


\section{The Radon transform outside a ball}
\label{secRadball}

Our goal in this section is to prove Theorem \ref{thmexteriorgene}. We will mostly study properties of the Radon transform on $L^2$ or $\dot H^1$. Throughout all this section,
\[ \text{we assume that the dimension } d \text{ is odd.} \]

We define the operator
\[ \m 1_{|s| \ge R}: L^2(\m R \times \m S^{d-1}) \to L^2(\m R \times \m S^{d-1}), \quad (\m 1_{|s| \ge R} f)(s,\omega) = \m 1_{|s| \ge R} f(s , \omega). \]
This is obviously an orthogonal projection. We will be interested in the operators
\[ \m 1_{|s| \ge R} \Td:  L^2(\m R^d) \to  L^2(\m R \times \m S^{d-1}) \quad \text{and} \quad \m 1_{|s| \ge R} \partial_s \Td:  \dot H^1 (\m R^d) \to  L^2(\m R \times \m S^{d-1}). \]

\begin{definition}
Denote the kernels
\[ K_{R}^{0}= \ker (\m 1_{|s|\ge R}\Td)\subset L^{2}(\R^{d}) \quad \text{and} \quad K_{R}^{1}= \ker (\m 1_{|s|\ge R}\partial_{s}\Td)\subset \dot H^{1}(\R^{d}) \]
 respectively, and $\pi_R$ and $\pi_R^1$ the orthogonal projections on $K_{R}$ and $K_{R}^{1}$ respectively.
\end{definition}

\begin{lem} \label{propproject}
Let $H, H'$ be two Hilbert spaces, $\phi : H \to H'$ a unitary operator (that is isometric and bijective) and $p$ an orthogonal projection on $H'$. Then, denoting $\pi: H \to H'$ the orthogonal projection on $\ker (p \phi)$, there hold
\[ \forall f \in H, \quad \| f \|_{H}^2 = \| (p \phi )(f) \|_{H'}^2 + \| \pi (f) \|_{H}^2. \]
\end{lem}

\begin{proof}
Consider $\phi^{-1} p \phi: H \to H$. One computes that it is an orthogonal projection with kernel $N = \ker (p \phi)$, that is $\phi^{-1} p \phi = 1-\pi$. Therefore, Pythagorean theorem yields that
\[ \forall f \in H, \quad \| f \|_H^2 = \| (\phi^{-1} p \phi)(f) \|_H^2 + \| (1-\phi^{-1} p \phi)(f) \|_{H}^2. \]
Now, $\phi^{-1}: H' \to H$ is isometric, so that $ \| (\phi^{-1} p \phi)(f) \|_H =  \| (p \phi)(f) \|_{H'}$, and $1-\phi^{-1} p \phi =\pi$ so that the above equality writes
\[ \forall f \in H, \quad \| f \|_H^2 = \| (p \phi)(f) \|_{H'}^2 + \| \pi(f) \|_{H}^2. \qedhere \]
\end{proof}

As a direct consequence of the above lemma and of Lemma \ref{lem:range}, we get that
\begin{align}
\label{proj:L2}
\forall u_1 \in L^2(\m R^d), \quad \| u_1 \|_{L^2}^2 & = \| \m 1_{|s|\ge R}\Td u_1 \|_{L^2(\m R \times \m S^{d-1})}^{2} + \| \pi_R^{0} u_1 \|_{L^2(\m R^d)}^2, \\
\label{proj:H1} \forall u_0 \in \dot H^1(\m R^d), \quad \| u_0 \|_{\dot H^1}^2 & = \| \m 1_{|s|\ge R}\partial_s \Td u_0 \|_{L^2(\m R \times \m S^{d-1})}^{2} + \| \pi_R^1 u_0 \|_{\dot H^1 (\m R^d)}^2. 
\end{align}

Our main goal in this paragraph will be to give explicit expressions of the kernels $K_R^{0}$ and $K_R^1$, and to relate them to the space $P(R)$ defined in the introduction. We emphasize that for this, we will make an essential use that $d$ is odd.

\bigskip
 
The main object of this section is to obtain the following theorem. We denote
\[  \dot H^s(|x| \le R) = \{ f \in \dot H^s(\m R^d) ; \Supp f \subset \overline{B(0,R)} \}, \]
and recall \eqref{def:Yl} that 
$(Y_\ell)_{\ell \in \m M}$ is an orthonormal basis of spherical harmonics, $l=l(\ell)$ is the degree of $Y_\ell$ and 
\[ \alpha_k = -l-d+2k \]
(it also depends on $\ell$), and we defined in the introduction the functions $f_k$ (adapted to the $\dot H^1$ context) and $g_k$ (adapted to the $L^2$ context), see \eqref{def:gk}-\eqref{def:fk}.

\begin{thm}

\label{thmKer}
Assume $d$ is odd. Then
\begin{align}
\label{decomp_ker_L2}
K_{R}^{0}=  L^2(|x| \le R) \stackrel{\perp}{\oplus} \Nd_{R}^{0} \quad \text{where} \quad \Nd_{R}^{0}= \bigoplus^{\perp}_{\ell\in \M}\Nd_{R, \ell}^{0}
\end{align}
(here $\perp$ means $L^2$-orthogonality) and
\begin{align*}
\Nd_{R, \ell}^0 = \Span \left(g_k ; \quad  k \in\N^{*}, \ \alpha_{k}< - d/2 \right).
\end{align*}

Similarly, there hold
\begin{align}
\label{decomp_ker_H1} 
K_R^1 = \dot H^1(|x| \le R) \stackrel{\perp}{\oplus} \Nd_{R}^{1}  \quad \text{where} \quad \Nd_{R}^{1}= \bigoplus^{\perp}_{\ell\in \M}\Nd_{R, \ell}^{1},
\end{align}
(here $\perp$ means $\dot H^1$-orthogonality)  and
\begin{align*}
\Nd_{R, \ell}^1 = \Span \left( f_k ; \quad  k\in\N^{*}, \ \alpha_{k}<1-d/2\right).
\end{align*}
\end{thm}

\begin{remark}
The Kernel of the partial Radon transform has already been computed by Quinto \cite[(3.14)]{Quinto:83} in different (weighted) spaces, namely $L_{p}^{2}(E)$, defined by its norm $\nor{f}{E,p}=\sqrt{2}\nor{|x|^{\frac{d-1}{2}}(1-|x|^{-2})^{p/2}f}{L^{2}(\{|x|\ge 1\})}$. He proves there (\cite[Corollary 3.4]{Quinto:83}), that the null space of $\m 1_{|s| \ge 1} \q R: L_{p}^{2}(E)\rightarrow L_{p}^{2}(E')$ is the closure of the span $|x|^{-d-k}Y_{\ell}(x/|x|)$ where $0 \le k<l$ and $k-l$ is even.

We however do not relie on this result, and actually use a different approach of proof.
\end{remark}

We will first consider the $L^2$ case, that is prove \eqref{decomp_ker_L2}, and then treat the $\dot H^1$ case for which the proof is analoguous, and we will only highlight the differences.

\begin{proof}[Proof of \eqref{decomp_ker_L2}]

As the dilation $f \mapsto \frac{1}{R^{d/2}} f(\cdot/R)$ is an isometry on $L^2(\m R^d)$, we can assume without loss of generality that $R=1$.

\bigskip

\emph{Step 1: Reduction to spherical harmonics.}

We define $\Nd_{1}^{0}$ to be the $L^2$-orthogonal complement of $L^2(|x| \le 1)$ in $K_1^0$:
\[ K_1^0 =: L^2(|x| \le 1) \stackrel{\perp}{\oplus} \Nd_{1}^{0}, \]
so that we have the explicit description
\[  \Nd_{1}^{0} = \left\{ f  \in L^2(\m R^d); f =0 \text{ on } \{ |x| \le 1 \} \text{ and } \Td f =0 \text{ on } \{ |s| > 1 \} \right\}. \]

It is convenient to introduce the following notation: if $w : \m R \to \m R$ and $Y : \m S^{d-1} \to \m R$ are two functions, then we define
\[ w \otimes Y : \m R^d \to \m R, \quad x \mapsto w(|x|) Y \left( \frac{x}{|x|} \right). \]

As all the functions we consider have symmetry in the $s$ variable, we keep track of it in the following definitions. For this, we denote
\[ L^2_{rad,l} = \left\{ w \in L^2(\m R, |r|^{d-1} dr) ; \forall r \in \m R \text{ a.e.}, \ w(-r) = (-1)^{l} w(r)  \right\}, \]
which we endow with the natural Hilbert norm:
\[ \| w \|_{L^2_{rad}} : =  |\m S^{d-1}|^{1/2}\| w \|_{L^2([0,+\infty), r^{d-1} dr)}. \]
The symmetry we impose on functions $w \in L^2_{rad,l}$ is essentially technical (the information required is given for $r \ge 0$); it is given for coherence purpose with the definition of $\q R_l$ (in Proposition \ref{lmRadldual}), mostly in Step 2 below. Then, for $\ell \in \m M$, let
\[ L^2_\ell := \{ f \in L^2(\m R^d); \exists w \in L^2_{rad,l},   f = w \otimes Y_\ell \}, \]
so that
 \[ L^2(\m R^d) = \bigoplus_{\ell \in \m M}^\perp L^2_\ell, \]
and the map $L^2_{rad} \to L^2_\ell$, $w \mapsto  w \otimes Y_\ell$ is a bijective isometry up to a constant:
\[ \|  w \otimes Y_\ell \|_{L^2} = \| Y_\ell \|_{L^2(\m S^{d-1})} \| w \|_{L^2_{rad}}. \]
The main point of this step is that $\Td$ preserve the structure in $L^2_\ell$. More precisely, denote
\[ \Td_l := c_0 (-1)^{\frac{d-1}{2}}\partial_s^{\frac{d-1}{2}} \q R_l, \]
then due to Proposition \ref{lmRadldual}, $\Td_l$ can be extended to an isometry from $L^2_{rad,l}$ to $L^2(\m R)$ (and arguing as in Lemma \ref{lem:range}, it is actually bijective):
\begin{align} \label{Tl_iso}
\forall w \in L^2_{rad}, \quad \| \Td_l w \|_{L^2} = \| w \|_{L^2_{rad}},
\end{align}
and we have the formula 
\[ \forall w \in L^2_{rad,l}, \quad \Td (w \otimes Y_\ell) =  \Td_l(w) \otimes Y_\ell. \]

We will now fix $\ell \in \m M$ and study the kernel
\[ \Nd_{1,\ell}^0 := \ker (\m 1_{|s| \le 1} \Td_l) = \{ w \in L^2_{rad,l} ; w \otimes Y_\ell \in  \Nd_{1}^{0} \}, \]
so that
\[ \Nd_{1}^0 = \bigoplus_{\ell \in \m M}^\perp (\Nd_{1,\ell}^0 \otimes \{ Y_{\ell} \}). \]

\bigskip

\emph{Step 2: $\Nd_{1,\ell}^0$ is finite dimensional.}

Let us first give an insight of the range of $\Td_\ell$ when restricted to $\Nd_{1,\ell}^0$.

\begin{lem}
\label{lmimaTpoly} Let $w\in \Nd_{1,\ell}^0$. Then there exists a polynomial $P$ such that $\ds \deg P \le l+\frac{d-5}{2}$ (with the convention that $\deg 0 = -\infty$), and
\[ \forall s \in \m R, \quad (\Td_{l}w)(s) = \m 1_{|s|\le 1}P(s). \] 
Also $P$ has the parity of $\ds l+\frac{d-1}{2}$.
\end{lem}

\begin{proof}
By definition, $q(s)=\Td_{l}w$ is an $L^{2}(\R)$ function supported on $[-1,1]$. We would like to use formula \eqref{formuleGegen} but we have to be careful of integrability issues, so we work by duality instead. We prove that $q^{(k)}=0$ on $]-1,1[$ in the sense of distributions, for $k\ge l+\frac{d-3}{2}$, and proceed via smooth approximations. 

Let $\varphi \in \q D (]-1,1[)$. For $\e>0$, let $\chi\in \q D(\{|r| > 1\})\cap L^2_{rad,l}$ such that $\nor{w-\chi}{L^{2}_{rad}}\le \e$. Since $\chi\in L^2_{rad,l}$ is smooth and compactly supported far from zero, it can be written $\chi(s)= r^{l} v(r^{2})$ for some $v\in \q D(\R)\subset \q S(\m R)$. In particular, we can apply Lemma \ref{lmRadl} to compute $\Td_{l}\chi$ and its derivatives. Using that $\chi$ is supported in $\{ |r| > 1 \}$, formula \eqref{formuleGegen} gives for $|s|<1$:
\begin{align*} 
(\Ra_{l}\chi )(s)& = \frac{|\m S^{d-1} |}{C_{l}^{\lambda}(1)}\int_{|s|}^{+\infty} C_{l}^{\lambda}\left(\frac{s}{r}\right)\chi(r)r^{d-2}\left(1-\frac{s^{2}}{r^{2}}\right)^{\frac{d-3}{2}}dr\\
& = \frac{|\m S^{d-1} |}{C_{l}^{\lambda}(1)}\int_{1}^{+\infty} C_{l}^{\lambda}\left(\frac{s}{r}\right)\chi(r)r^{d-2}\left(1-\frac{s^{2}}{r^{2}}\right)^{\frac{d-3}{2}}dr.
\end{align*} 
We now differentiate $k + \frac{d-1}{2}$ times (using that $\Ra_{l}\chi$ is a smooth function in $]-1,1[$) to obtain
\begin{align*} 
\left( \frac{d^{k}}{ds^{k}} \Td_{l}\chi \right)(s)=c_{0}(-1)^{\frac{d-1}{2}}\frac{|\m S^{d-1} |}{C_{l}^{\lambda}(1)}\int_{1}^{+\infty} \chi(r)r^{d-2}\frac{d^{k}}{ds^{k}}\partial_{s}^{\frac{d-1}{2}} \left[C_{l}^{\lambda}\left(\frac{s}{r}\right)\left(1-\frac{s^{2}}{r^{2}}\right)^{\frac{d-3}{2}}\right]dr
\end{align*} 
Since $C_{l}^{\lambda}$ is a polynomial of order $l$ and $\left(1-\frac{s^{2}}{r^{2}}\right)^{\frac{d-3}{2}}$ is a polynomial of order $d-3$ ($d$ is odd!), the right hand side is zero if $k+\frac{d-1}{2}> l+d-3$. In particular,  for $k\ge l+\frac{d-3}{2}$
\[ \forall s \in [-1,1], \quad \left( \frac{d^{k}}{ds^{k}} \Td_{l}\chi \right)(s)=0. \]
As $\Td_{l}$ is isometric (see \eqref{Tl_iso}), we have 
 \[ \nor{q-\Td_l \chi}{L^{2}_{sym}}=\nor{w-\chi}{L^{2}_{rad}} \le \e. \]
 Since $q=0$ for $|s|\ge 1$, there holds $\nor{\Td_l \chi}{L^{2}_{sym}(|s|\ge 1)} \le \e$ and 
\begin{align*} 
\left|\int_{\R} \varphi^{(k)}(s) q(s)ds\right|& \le \left|\int_{|s| < 1} \varphi^{(k)}(s) \left[q(s)-\Td_l \chi\right]ds\right|+\left|\int_{|s| < 1} \varphi^{(k)}(s) \Td_l \chi ds\right|\\
& \le 2 \e \nor{\varphi^{(k)}}{L^{2}}.
\end{align*} 
Therefore, we obtain that $\ds \int_{\R} \varphi^{(k)}(s) q(s)ds=0$ for any $\varphi \in \q D (]-1,1[)$. Hence, $q$ is a polynomial of degree less or equal to $l+\frac{d-5}{2}$ on $]-1,1[$. Since $q$ is a $L^{2}$ function that is zero on $\{|s|\ge 1\}$, it gives the result.
\end{proof}

\begin{remark}
It is likely that the previous method applies well to other spaces like $H^{-\rho}$ (as the Radon transform was extended to these spaces), as long as its elements can be approximated by functions with compact support in $\{ |s| \ge 1 \}$.
\end{remark}

\begin{cor}
$\Nd_{1,\ell}^0$ is finite dimensional, of dimension $K_{\ell}\le  \left\lfloor \frac{l}{2}+\frac{d-1}{4} \right\rfloor$.
\end{cor}

\begin{proof}
The space of symmetric polynomials of degree at most $m$ has dimension $\frac{m}{2}+1$ or $\frac{m+1}{2}$ depending on the parity of $m$ and even/odd polynomial; in any case, it is at most $\left\lfloor \frac{m}{2} + 1 \right\rfloor$. Lemma \ref{lmimaTpoly} thus implies that $\Td_{l}\Nd_{1,\ell}^0$ is contained in a finite dimensional subspace of dimension less than $\left\lfloor \frac{l}{2}+\frac{d-1}{4} \right\rfloor$. Since $\Td_l$ is an isometry on its Range, as seen in Step 1, $\Nd_{1,\ell}^0$ is therefore finite dimensional with the same dimension.
\end{proof}

\bigskip

\emph{Step 3: $\Nd_{1,\ell}^0$ is spanned by functions of the type $\ln(|r|)^p r^\alpha$, $\alpha \in \m C$, $p \in \m N$.}

We now have some precise information about the image of $\Nd_{1,\ell}^0$ by $\Td_{l}$, so that it only remains to invert it. Since $\Td_l$ is invertible in the appropriate $L^2$-related spaces, it might be possible to directly use Lemma \ref{lmRadldual} to recover $\Nd_{1,\ell}^{0}$ by applying the inverse of $\Td_\ell$ to functions which are the product of a polynomial by an indicatrix function. Yet, we prefer to apply homogeneity arguments that yield directly the result that $\Nd_{1,\ell}^0$, being finite dimensional, can only contain the restriction of homogeneous distributions. 

\begin{lem}
For any $\lambda>1$, define the dilation/restriction operator $S_{\lambda}$ acting on  functions $w: \m R \to \m R$ by  
\begin{align}  
\label{Alam}
\forall r \in \m R, \quad (S_{\lambda} w)(r) := \m 1_{|r|> 1} w(\lambda r).
\end{align} 
Then for any $\ell \in \m M$, $S_\lambda$ maps $\Nd_{1,\ell}^0$ into itself.
\end{lem}

\begin{proof}
Let $w \in\Nd_{1,\ell}^0$, and consider $v =  (S_\lambda w) \otimes Y_\ell$, so that $\Td_l (S_\lambda w) \otimes Y_\ell = \Td(v)$.

Observe that 
\[ v = \m 1_{|x| \ge 1} ((w \otimes Y_\ell) \circ M_\lambda) = ( \m 1_{|x| \ge \lambda} (w \otimes Y_\ell)) \circ M_\lambda. \]
Therefore, using Lemma \ref{lmRahomogen}, adapted to $\Td$ instead of $\partial_{s}\Td$, we infer
\[ \Td (v)=\lambda^{-\frac{d-1}{2}}\left(\Td( \m 1_{|x| \ge \lambda} (w \otimes Y_\ell) )\right) \circ M_{\lambda}. \]
Now, let $\varphi\in \q D(\R\times \m S^{d-1})$ so that $\varphi$ is supported in $\{|s|> 1\}$. We get
\begin{multline*} 
 \left< \Td (v),\varphi \right>_{\mathcal{S}'(\R\times \m S^{d-1}),\mathcal{S}(\R\times \m S^{d-1})} \\
 = \lambda^{-\frac{d-1}{2}-1}\left< \Td( \m 1_{|x| \ge \lambda} (w \otimes Y_\ell) ),\varphi \circ M_{1/\lambda}\right>_{\mathcal{S}'(\R\times \m S^{d-1}),\mathcal{S}(\R\times \m S^{d-1})}.
  \end{multline*} 
The assumption on $\Supp(\varphi)$ implies that $\varphi\circ M_{1/\lambda}$ is supported in $\{|s|> \lambda\}$. Applying Proposition \ref{propdsTH1} to $\m 1_{|x| \ge \lambda } (w \otimes Y_\ell)$, we get that 
\[ (\Td (\m 1_{|x| \ge \lambda } (w \otimes Y_\ell))_{\left|\{|s|> \lambda\}\times \m S^{d-1}\right.}=(\Td (w\otimes Y_\ell))_{\left|\{|s|> \lambda\}\times \m S^{d-1}\right.}=0, \]
 since $w \in \Nd_{1,\ell}^0$ and $\lambda>1$. So, we have proved that for any test function $\varphi$ supported in $\{|s|> 1\}$
 \[ \left< \Td (v),\varphi \right>_{\mathcal{S}'(\R\times \m S^{d-1}),\mathcal{S}(\R\times \m S^{d-1})}=0, \]
in other words, $\Td (v)=0$ on $\{|s|> 1\}$:  hence $\m 1_{|s| > 1} \Td_l(S_\lambda w) =0$ and $S_{\lambda} w \in \Nd_{1,\ell}^0$. 
\end{proof}

We now state a general fact, which describes finite dimensional spaces of 1D functions invariant by scaling.

\begin{lem}
\label{lmfunctionel}
Let $N\subset L^{1}_{loc}([1,+\infty))$ be a \emph{finite dimensional} vector space such that  for any $\lambda>1$, $S_{\lambda}(N) \subset N$. Then, there exist a finite set $I$, $(\alpha_{i})_{i\in I} \subset \m C$, $(p_{i})_{i \in I} \subset \N$  so that 
\[ N=\Span \left( r \mapsto \log(r)^{j}r^{\alpha_{i}}; i\in I, 0\le j\le p_{i}-1 \right). \]
\end{lem}

\begin{proof}
Notice that all the $(S_{\lambda})_{\lambda>1}$ are commuting applications: $S_{\lambda}S_{\beta}=S_{\lambda \beta}$. Also, in the logarithmic variable $s=\log(r)$, $s\ge 0$, $S_{\lambda}$ is the translation with generator the derivation. That is, if $w \in N$ and $v: s \mapsto w(e^{s})$, we have with this representation $S_{e^{\beta}} v(s)=1_{s\ge 0} v(s+\beta)$, and this defines a semigroup. Denote $A$ the infinitesimal generator of $\beta\mapsto S_{e^{\beta}}$. Choose a basis of $N$ so that $A$ has a Jordan form: it is block diagonal block and each diagonal block (of size say $p+1$) takes the form
\begin{align*} 
J=\begin{bmatrix}
\alpha & 1 &  &   &  & \\
    & \alpha & 1 &  &(0)  &  \\
		 &  &\alpha &1  &  & \\
		 & (0)   & & \ddots  &\ddots  &  \\
  & & & &\ddots  &1 \\
     & &  &   &  & \alpha
\end{bmatrix} 
\quad \text{so that} \quad
e^{\lambda J}=e^{\lambda \alpha}
\begin{bmatrix}
1& \lambda &\frac{\lambda^{2}}{2}  &\cdots   &  &\frac{\lambda^{p}}{p!} \\
    &1 & \lambda &  & &  \\
		 &  &1 &\lambda  &  & \\
		 & (0)   & & \ddots  &\ddots  &  \\
  & & & &\ddots  &\lambda \\
     & &  &   &  & 1
\end{bmatrix}.
\end{align*} 
In particular, in this base $(g_{0},\cdots,g_{p})$ we can write for any $s, \lambda \ge 0$
\begin{align*} 
g_{0}(s+\lambda)& = e^{\lambda \alpha}g_{0}(s), \\
g_{1}(s+\lambda)& = \lambda e^{\lambda \alpha}g_{0}(s)+ e^{\lambda \alpha}g_{1}(s), \\
  & \vdots \\
g_{p}(s+\lambda)& = \frac{\lambda^{p}}{p!} e^{\lambda \alpha}g_{0}(s)+ \cdots+  e^{\lambda \alpha}g_{p}(s).
\end{align*}
Taking $s=0$ in this equalities gives
\[ g_j(\lambda) = \frac{\lambda^{j}}{j!} e^{\lambda \alpha} g_0(0) + \cdots + \cdots +e^{\lambda} g_j(0), \]
that is denoting $f_j: s \mapsto s^j e^{\alpha s}$, there hold
\[ \begin{bmatrix}
g_0(0) & g_1(0) & g_2(0) &   &  & g_p(0) \\
    & g_0(0) & g_1(0) &  &  &  \\
		 &  &g_0(0) & g_1(0)  &  & \\
		 & (0)   & & \ddots  &\ddots  &  \\
  & & & &\ddots  &g_1(0) \\
     & &  &   &  & g_0(0)
\end{bmatrix}   \begin{bmatrix} 
f_0 \\ f_1 \\  f_2 \\ \vdots \\ f_{p-1} \\ f_p
\end{bmatrix} = \begin{bmatrix} 
g_0 \\ g_1 \\ g_2 \\ \vdots \\ g_{p-1} \\ g_p
\end{bmatrix}. \]
Observe that $g_0(0) \ne 0$ (otherwise $g_0 \equiv 0$ which would contradict $(g_j)_j$ being a base), so that the above matrix is invertible, and the $(f_j)_{0 \le j \le p}$ form a base of each block of the Jordan base of $A$ in $N$. We get the result getting back to the original variable $r=e^{s}$.
\end{proof}

Gathering together the above two results, we infer that $\q N_{1,\ell}^0$ admit a base made of functions $w \in L^2_{rad,l}$ such that for $r \ge 0$, 
\begin{align} \label{def:w_ja}
w(r) = \m 1_{r \ge 1} \ln(r)^p r^\alpha
\end{align}
for some $p \in \m N$ and $\alpha \in \m C$. Denote $\q B$ the set of couples $(p,\alpha) \in \m N \times \m C$ which appear in this base: due to Lemma \ref{lmfunctionel}, $\q B$ is a finite union of $\{ (0,\alpha_i),(1,\alpha_i), \dots, (p_i,\alpha_i) \}$. 

Finally, we state a second stability result, to be used in the following step.

\begin{lem} \label{N_stab_Delta}
Consider the operator $\check \Delta: f \mapsto \m 1_{|x| > 1} \Delta f$. 
Then $\check \Delta$ (is well defined and) maps $ \q N_{1}^0 \cap L^2_{\ell}$ to itself: for all $f = w \otimes Y_\ell \in \q N_{1}^0 \cap L^2_{\ell}$, we also have $\m 1_{|x| > 1} \check \Delta f \in \q N_{1}^0 \cap L^2_{\ell}$.
\end{lem}

This is of course strongly connected to the explicit special form of $w \in \q N_{1,\ell}^0$.

\begin{proof}
For $w$ as in \eqref{def:w_ja}, a direct computation yields that $\m 1_{|x| >1} (\Delta (w \otimes Y_\ell)) \in L^2$ (it is a smooth function on $\{ |x| >1 \}$) , and recalling the form of the Laplacian in spherical coordinates, this function belongs to $L^2_\ell$. Now, from \eqref{actionDeltaRadonDhs}, we also have that
\[ \Td \Delta(w \otimes Y_\ell) = \partial_s^2 \Td(w \otimes Y_\ell), \]
so that $\m 1_{|s| >1} \Td \Delta(w \otimes Y_\ell)  =0$. By Proposition \ref{propdsTH1}, $\Td (\m 1_{|x| >1} \Delta(w \otimes Y_\ell)$ and $\Td \Delta(w \otimes Y_\ell)$ coincide on $\{ |s| >1 \}$, and are both $0$ there. Hence $\m 1_{|x| > 1}  \Delta(w \otimes Y_\ell) \in \q N_1^0$.
\end{proof}

\bigskip

\emph{Step 5: $\q N_1^0 \cap L^2_\ell$ is spanned by the $g_k$.}

We recall the following formula,  valid for $p \in \m N$ and $\alpha \in \m C$: for $x \ne 0$,
 \begin{multline}  
\label{formuLaplpowerlog}
\Delta \left[\log(|x|)^{p}|x|^{\alpha}Y_{\ell} \left( \frac{x}{|x|} \right) \right]  = \left[\left[\alpha (\alpha+d-2)-l (l+d-2)\right]\log(|x|)^{p} \right. \\
 \left. +p(2\alpha+d-2) \log(|x|)^{p-1}+p(p-1)\log(|x|)^{p-2} \right]  |x|^{\alpha-2}Y_{\ell} \left( \frac{x}{|x|} \right) 
\end{multline} 
(with the convention that $|x|^0=1$; for the convenience of the reader, a derivation of this formula is presented in Appendix \ref{app:a}). We now claim:

\begin{lem} \label{lem:B_0}
Let $(p,\alpha) \in \q B$. Then $p=0$, $\alpha < -d/2$ and $\alpha = -l - d + 2(k+1)$ for some $k \in \m N$.
\end{lem}

\begin{proof}
Denote $w \in \q N_{1,\ell}^0$ such that
\begin{align*} 
\forall r \ge 1, \quad w(r) = \ln(r)^p r^\alpha.
\end{align*}

1) As $w \in L^2_{rad,l}$,  $\alpha < -d/2$. 

\bigskip

2) We next prove that $\alpha = -l - d + 2(k+1)$ for some $k \in \m N$.

As $\q N_{1}^0 \cap L^2_{\ell}$ is isometric to $\q N_{1,\ell}^0$, it is finite dimensional. On the other hand, it is stable by $\check \Delta$, due to Lemma \ref{N_stab_Delta}. Therefore, for all $k \in \m N$, $\check \Delta^k (w \otimes Y_\ell) \in \q N_{1}^0 \cap L^2_{\ell}$. Now, as $\check \Delta$ has the same action as $\Delta$ for $\{ |x| >1 \}$, direct computations which follow from \eqref{formuLaplpowerlog} give that for $|x| > 1$,
\[ \tilde \Delta^k (w \otimes Y_\ell) (x) = C_{k}\log(|x|)^{p}|x|^{\alpha-2k}Y_{\ell}(\frac{x}{|x|}) + \sum_{j=0}^{p-1} c_{j,k}\log(|x|)^{j}|x|^{\alpha-2k}Y_{\ell}(\frac{x}{|x|}), \]
for some coefficients $c_{j,k} \in \m C$ and where $C_0=1$ and by induction, for $k \in \m N$,
\[ C_{k+1} = ((\alpha-2k)(\alpha-2k + d-2) - l (l+d-2)) C_k. \]
However the functions $(r \mapsto \log(r)^{j}|x|^{\alpha-2k})_{k \in \m N}$ are linearly independent (recall $\alpha < -d/2$), so that the $(\tilde \Delta^k (w \otimes Y_\ell))_k$ too: as $\q N_{1}^0 \cap L^2_{\ell}$ is finite dimensional, it implies that for some $k \in \m N$, $C_{k+1}=0$, in other words,
\[ (\alpha-2k)(\alpha-2k + d-2) - l (l+d-2) =0. \]
As $\alpha -2k <0$, we infer that $\alpha-2k = -l -d+2$.

\bigskip

3) Let us finally prove that $p=0$. We argue by contradiction and assume that $p \in \m N^*$. Then, thanks to the structure of $\q B$ already precised, we have $(1,\alpha) \in \q B$. Without loss of generality we can furthermore assume that $\alpha$ is minimal for this property. Now, we compute that for $|x| > 1$
\begin{multline*} 
\Delta \left[ \log(|x|)|x|^{\alpha}Y_{\ell} \left( \frac{x}{|x|} \right) \right] \\
=\left[\left[\alpha (\alpha+d-2)-l (l+d-2)\right]\log(|x|)+(2\alpha+d-2) \right] |x|^{\alpha-2}Y_{\ell} \left( \frac{x}{|x|} \right).
\end{multline*} 
 If $\alpha \ne -l -d+2$, by linear independence, we would have $(1,\alpha-2) \in \q B$, which would contradict the minimality of $\alpha$. Hence $\alpha = - l - d+2$ and as $\alpha <-d/2$, $2\alpha +d-2 <0$ is not null, so that $(0,\alpha-2) = (0,-l -d) \in \q B$. Applying $\check \Delta$ et using \eqref{formuLaplpowerlog} repetitively, we would get $(0, -l -d-2k) \in \q B$ for all $k \in \m N$, which contradicts that $\q B$ is finite. Therefore $p=0$.
\end{proof}

Let $N_\ell \in \m N$ be the maximum of the $k$ such that $(0,-l-d+2k) \in \q B$. Then by applying repetitively $\check \Delta$ (and using \eqref{formuLaplpowerlog}) to $w \otimes Y_\ell$ where $w \in L^2_{rad,l}$ and $w(r) = r^{N_\ell}$ for $r \ge 1$, we get that for all $k \in \llbracket 1, N_\ell \rrbracket$, $(0,-l-d+2k) \in \q B$.
Recalling the definition of the $g_k$ \eqref{def:gk}, we can reformulate this by saying that
\begin{align} \label{N10_span_1}
\q N_{1}^0 \cap L^2_\ell = \Span (g_k; k \in \llbracket 1, N_\ell \rrbracket).
\end{align}

\bigskip

\emph{Step 6: Conclusion.}

We now complete the description of $\q N_{1}^0 \cap L^2_\ell$, that is we prove that
\[ \q N_{1}^0 \cap L^2_\ell = \Span (g_k; k \in \m N, \alpha_k < -d/2). \]
For this, it suffices to prove that if $\alpha_k < -d/2$, then $g_k \in \q N_{1}^0 \cap L^2_\ell$.

It is certainly possible to prove it by direct computation using the formulae of Lemma \ref{lmRadldual} and \ref{formuleGegendual}. Similar computations are made for instance in Quinto \cite[Formula (3,14)]{Quinto:83} (see also \cite[Section 7.3-7.4 p795]{GR:2007} for related computations). Yet, it is not easy (but certainly doable) to justify the computations when the functions do not have enough decay. 

Instead, we will use some ideas related to Lemma \ref{N_stab_Delta}.

\begin{lem}
Let $k \in \m N$ such that $\alpha_k < -d/2$. Then $g_k \in \q N_{1}^0 \cap L^2_\ell$.
\end{lem}

\begin{proof}
Observe that, as $\alpha_k < -d/2$, $g_k \in L^2_\ell$.

 Since $\Delta$ is a differential operator, we have $(\Delta g_{k})_{\{ |x| > 1 \}}=\Delta (g_{k})_{\{ |x| > 1 \}}$ in the sense of distribution (the $\Delta$ being an operator on distributions either in $\R^{d}$ or in $  \{ |x| > 1 \}$).
$(g_{k})_{\{ |x| > 1 \}}$ is a smooth function, so formula \eqref{formuLaplpowerlog} (with $p=0$), applies for $k \ge 1$, in the classical sense, to give\[\Delta (g_{k})_{\{ |x| > 1 \}}= c_k (g_{k-1})_{\{ |x| > 1 \}},   \]
for some $c_k \in \m R$, either in the classical sense or in the distributional sense in $\mathcal{D}'(|x| > 1)$. Thanks to the previous remark, we obtain
\[(\Delta g_{k})_{\{ |x| > 1 \}}= c_k (g_{k-1})_{\{ |x| > 1 \}}  \]
Using now part \eqref{supporprop} of Proposition \ref{propdsTH1}, it gives
\[ (\Td ( \Delta g_k))|_{\{ s > 1 \}} =  c_{k }(\Td g_{k-1})_{\{ s > 1 \}}. \]
Using this time part \eqref{itemactionDeltaRadonDhs} of Proposition \ref{propdsTH1} and after restricting to $\{ s > 1 \}$, there holds
\[ (\Td ( \Delta g_k))|_{\{ s > 1 \}} = (\partial_s^2 \Td g_k)_{\{ s > 1 \}}. \]
As $g_k \in L^2_\ell$, denote $h_k \in L^2(\{ |s| > 1 \})$ such that $(\Td g_k)|_{\{|s| > 1\}} = h_k \otimes Y_\ell$: we obtained for $k \ge 1$
\begin{align} \label{diff:hk}
\partial_s^2 h_k = c_k h_{k-1}.
\end{align}
Similarly, and by definition of $\alpha_0$ and spherical harmonic, we get
\[\Delta (g_{0})_{\{ |x| > 1 \}}= 0.  \]
%

Let us prove by induction on $k$ (such that $\alpha_k <-d/2$) that $h_k=0$.

For $k=0$, the function $u(t,x)= g_0(x)$ is solution of $\Box u=0$ on  $\{ |x| > |t|+1 \}$ (outside of the light cone). In particular, by finite speed of propagation, the solution $w$ to the wave equation with initial data $(g_0,0)$, satisfies $w(t,x)=g_{0}(x)$ on $\{ |x| > |t|+1 \}$. In particular, \eqref{equivfinallemmaT} gives 
\begin{align*}
0=2 \lim_{t \to +\infty} \nor{w}{L^2(|x| \ge t+1)}^2  & =  \nor{\Td g_0}{L^2(\{s \ge 1\}\times \m S^{d-1})}^2,
\end{align*}
This proves that $h_0=0$.

Let $k \ge 1$ and assume that $h_{k-1} =0$. In view of \eqref{diff:hk}, we infer that there exists $\alpha, \beta \in \m R$ such that $h_k(s) = \alpha s+ \beta$ for $s >1$. As $h_k \in L^2(\{ |s| > 1 \},ds)$, $\alpha=\beta =0$ and $h_k =0$. This completes the induction.

Now, $h_k=0$ precisely means that $\m 1_{|s| \ge 1} \Td g_k =0$, and so $g_k \in \q N_1^0$.
\end{proof}

This completes \eqref{decomp_ker_L2}, that is the proof of Theorem \ref{thmKer} in the $L^2$ case.
\end{proof}

\begin{proof}[Proof of \eqref{decomp_ker_H1}]

 As $g \mapsto \frac{1}{R^{d/2-1}} g(\cdot/R)$ is an isometry on $\dot H^1(\m R^d))$, we can assume as before that $R=1$. In this proof, all orthogonalities are meant with respect to the $\dot H^1$ scalar product.

\bigskip

\emph{Step 1: Reduction to spherical harmonics}

We now define $\q N_1^1$ to be the $\dot H^1$-orthogonal complement of $\dot H^1(|x| \le 1)$ in $K_1^1$:
\[ K_1^1 = \dot H^1(|x| \le 1) \stackrel{\perp}{\oplus} \q N_1^1. \]
The explicit description of $ \q N_1^1$ is now different from the $L^2$, as it involves the harmonic extension of in $B(0,1)$: 
\[ \q N_1^1 = \left\{ f \in \dot H^1(\m R^d); \Delta f =0 \text{ on } \{ |x| < 1 \} \text{ and } \partial_s \Td f =0 \text{ on } \{ |s| > 1 \} \right\}. \]

Analoguously to the $L^2$ case, we define
\[ \dot H^1_{rad,l} = \left\{ w \in \dot H^1(\m R, |r|^{d-1} dr); \forall r \in \m R \text{ a.e.}, \ w(-r) = (-1)^{l+1} w(r) \right\}. \]
This time, we equip $\dot H^1_{rad,l}$ with a family of norm which are all equivalent, but adapted to the $Y_\ell$: 
\[ |\m S^{d-1}|^{-1} \| w \|_{\dot H^1_{rad,l}}^2 :=  \| \partial_r w \|_{L^2([0,+\infty), |r|^{d-1} dr)}^2 + l(l+d-2) \left\| \frac{w}{r} \right\|_{L^2([0,+\infty), |r|^{d-1} dr)}^2. \]
Then, we can define
\[ \dot H^1_\ell := \{ f \in \dot H^1 (\m R^d); \exists w \in \dot H^1_{rad,l},   f = w \otimes Y_\ell \}, \]
so that
 \[ \dot H^1(\m R^d) = \bigoplus_{\ell \in \m M}^\perp H^1_\ell, \]
 and the map $\dot H^1_{rad,l} \to \dot H^1 _\ell$, $w \mapsto  w \otimes Y_\ell$ is a bijective isometry up to a constant, for the right norm:
\[ \|  w \otimes Y_\ell \|_{\dot H^1} =\| w \|_{\dot H^1_{rad,l}}. \]
Again, $\partial_s \Td$ preserve this structure: $\partial_s \Td = c_0 (-1)^{\frac{d-1}{2}}\partial_s^{\frac{d+1}{2}} \q R_l$ can be extended to a (bijective) isometry from $\dot H^1_{rad,l}$ to $\dot H^1(\m R)$
\begin{align} \label{dTl_iso}
\forall w \in \dot H^1_{rad,l}, \quad \| \partial_s \Td_l w \|_{\dot H^1(\m R)} = \| w \|_{\dot H^1_{rad,l}},
\end{align}
and the commutative diagram still holds
\[ \forall w \in \dot H^1_{rad,l}, \quad \partial_s \Td (w \otimes Y_\ell) =  (\partial_s \Td_l w) \otimes Y_\ell. \]

As mentioned above, we will use the harmonic extension on $B(0,1)$, that is the operator $\q P$ such that, for $f \in \dot H^1$, $\q Pf$ satisfies $\q Pf(x)= f(x)$ for $|x| \ge 1$ and $\Delta \q P f =0$ on $B(0,1)$, so that
\[ \q N_1^1 = \left\{ f \in \Im(\q P); \partial_s \Td f =0 \text{ on } \{ |s| > 1 \} \right\}. \]
$\q P$ is actually the $\dot H^1$-orthogonal projector on $\dot H^1(|x| \le 1)^\perp$. Observe that the action of $\q P$ on $\dot H^1_\ell$ is simple: for $w \in \dot H^1_{rad,l}$,
\[ \q P(w \otimes Y_\ell)(x) = \begin{cases}
(w \otimes Y_\ell)(x)  = w(|r|) Y_\ell(x/|x|) & \quad \text{if } |x| \ge 1, \\
w(1) |x|^l Y_\ell(x/|x|)  &\quad \text{if } |x| < 1.
\end{cases} \]
In other words, we can define an operator $\q P_l:  \dot H^1_{rad,l} \to  \dot H^1_{rad,l}$, 
\[ w \mapsto \m 1_{|x| < 1} \sgn(x)^{l+1} |x|^\ell + \m 1_{|x| \ge 1} w, \]
so that for all $\ell \in M$ and $w \in \dot H^1_{rad,l}$, $\q P(w \otimes Y_\ell) = (\q P_l w) \otimes Y_\ell$.

We will now fix $\ell \in \m M$ and study the kernel
\begin{align*}
\Nd_{1,\ell}^1 & := \ker (\m 1_{|s| \le 1} \Td_l) = \{ w \in \dot H^1_{rad,l} ; w \otimes Y_\ell \in  \Nd_{1}^{1} \} \\
& = \left\{ w \in \dot H^1_{rad,l} ; w(r) = cr^l \text{ for } 0 \le r <1 \text{ and } \partial_s \Td_l w =0 \text{ on } \{ |s| \ge 1 \} \right\},
\end{align*}
and there hold
\[ \Nd_1^1= \bigoplus_{\ell \in \m M}^\perp ( \Nd_{1,\ell}^1 \otimes \{ Y_\ell \}). \]

\bigskip

\emph{Step 2: $\Nd_{1,\ell}^1$ is finite dimensional.}

\begin{lem}
\label{lmimaTpolyH1} Let $w\in \Nd_{1,\ell}^0$. Then there exist a polynomial $P$ such that $\ds \deg P \le l+\frac{d-3}{2}$ (with the convention that $\deg 0 = -\infty$), and
\[ \forall s \in \m R, \quad (\Td_{l}w)(s) = \m 1_{|s|\le 1}P(s). \] 
Also $P$ has the parity of $\ds l+\frac{d+1}{2}$.
\end{lem}

\begin{proof}
The $L^2$ proof adapts \emph{mutatis mutandis}, working on $q(s)=\partial_s \Td_{l}w$ which is in $L^{2}_{rad,l}$.
\end{proof}

\begin{cor}
$\Nd_{1,\ell}^0$ is finite dimensional, of dimension at most $\left\lfloor \frac{l}{2}+\frac{d+1}{4} \right\rfloor$.
\end{cor}

\begin{proof}
As in the $L^2$, it is consequence of $\partial \Td_l$ being an isometry and $\partial \Td_l(\q N_{1,\ell}^1)$ being finite dimensional, in view of the previous lemma.
\end{proof}

\bigskip

\emph{Step 3: $\Nd_{1,\ell}^0$ is spanned by functions of the type $\ln(|r|)^p r^\alpha$, $\alpha \in \m C$, $p \in \m N$.}

\begin{lem}
For any $\lambda>1$, define the dilation/restriction operator $\tilde S_{\lambda}$ acting on functions $w \in \dot H^1_{rad,l}$ by  
\begin{align}  
\label{Alam1}
(\tilde S_{\lambda} := (\q P_l (r \mapsto w(\lambda r)).
\end{align} 
Then for any $\ell \in \m M$, $\tilde S_\lambda$ maps $\Nd_{1,\ell}^1$ into itself.
\end{lem}

\begin{proof}
The proof follows the path of its $L^2$ counterpart with a little variation due to the harmonic extension. We also need the operator $\q P^\lambda$, which is defined as $\q P$ on $\dot H^1(\m R^d)$, but performs the harmonic extension on $B(0,\lambda)$ instead ($\q P^\lambda$ is a rescale of $\q P$, not to be confused with $\q P_l$, which is a quotient map of $\q P$).

Let $w \in\Nd_{1,\ell}^1$, and consider $v =  (\tilde S_\lambda w) \otimes Y_\ell$, so that $\partial_s \Td(v) = (\partial_s \Td_l(\tilde S_\lambda w) ) \otimes Y_\ell$.

The key point is that
\[ v = (\q P^\lambda (w \otimes Y_\ell)) \circ M_\lambda. \]
Now, using Lemma \ref{lmRahomogen}, we infer
\[ \partial_s \Td (v)=\lambda^{-\frac{d+1}{2}}\left(\partial_s \Td( \q P^\lambda (w \otimes Y_\ell) )\right) \circ M_{\lambda}. \]
Let $\varphi\in \q D(\R\times \m S^{d-1})$ so that $\varphi$ is supported in $\{|s|> 1\}$. We get
\begin{multline*} 
 \left< \partial_s \Td (v),\varphi \right>_{\mathcal{S}'(\R\times \m S^{d-1}),\mathcal{S}(\R\times \m S^{d-1})} \\
 = \lambda^{-\frac{d+1}{2}}\left< \Td( \q P^\lambda (w \otimes Y_\ell) ),\varphi \circ M_{1/\lambda}\right>_{\mathcal{S}'(\R\times \m S^{d-1}),\mathcal{S}(\R\times \m S^{d-1})}.
  \end{multline*} 
The assumption on $\Supp(\varphi)$ implies that $\varphi\circ M_{1/\lambda}$ is supported in $\{|s|> \lambda\}$. Also, $\q P^\lambda w =w$ on $\{|x| \ge \lambda$ so that, applying Proposition \ref{propdsTH1}, we get that 
\[ (\partial_s \Td (\q P^\lambda (w \otimes Y_\ell))_{\left|\{|s|> \lambda\}\times \m S^{d-1}\right.}=(\partial_s \Td (w\otimes Y_\ell))_{\left|\{|s|> \lambda\}\times \m S^{d-1}\right.}=0, \]
 since $w \in \Nd_{1,\ell}^0$ and $\lambda>1$. So, we have proved that for any test function $\varphi$ supported in $\{|s|> 1\}$
 \[ \left< \partial_s \Td (v),\varphi \right>_{\mathcal{S}'(\R\times \m S^{d-1}),\mathcal{S}(\R\times \m S^{d-1})}=0, \]
in other words, $\partial_s \Td (v)=0$ on $\{|s|> 1\}$:  hence $\m 1_{|s| \le 1} \partial_s \Td_l(\tilde S_\lambda w) =0$ and $\tilde S_{\lambda} w \in \Nd_{1,\ell}^1$. 
\end{proof}

Now observe that $\tilde S$ has the same action as $S$ for $|r| \ge 1$: for $w \in \dot H^1_{rad,l}$, and $r \ge 1$, $( \tilde S_\lambda w)(r) =w(\lambda r) = (S_\lambda w)(r)$, so that we can still use Lemma \ref{lmfunctionel} as is. We conclude that $\q N_{1,\ell}^1$ admits a base made of functions $w \in \dot H^1_{rad,l}$ such that for some $p \in \m N$ and $\alpha \in \m C$ and $r \ge 0$,
\begin{align} \label{def:w_ja_1}
w(r) = \begin{cases}  \ln(r)^p r^\alpha \text{ if } r \ge 1, \\
r^l \text{ if } 0 \le r < 1.
\end{cases}
\end{align}

Denote again $\tilde {\q B}$ the set of couples $(p,\alpha) \in \m N \times \m C$ which appear in this base: due to Lemma \ref{lmfunctionel}, $\tilde{\q B}$ is a finite union of $\{ (0,\alpha_i),(1,\alpha_i), \dots, (p_i,\alpha_i) \}$. 

The other stability result, suitably modified, also holds in the $\dot H^1$ context. The only subtlety is the definition of extension operator. We already defined the harmonic extension $\q P$, but it can easily be extended in the following way: if $f \in \q C(\{ |x| \ge 1 \})$ is continuous up to the boundary, one can consider the harmonic extension $g$ of $f|_\m S^{d-1}$ in $B(0,1)$:
\[ \Delta g =0 \text{ on } B(0,1), \quad g|_{\m S^{d-1}} = f |_{\m S^{d-1}}. \]
Then we still denote 
\[ \q Pf: x \mapsto \begin{cases}
g(x) & \text{ if } |x| \le 1, \\
f(x) & \text{ if } |x| \ge 1.
\end{cases} \]
If $w \in \q C([1,+\infty))$ then this is simply
\[ \q P(w \otimes Y_\ell) = \tilde w \otimes Y_\ell, \quad \text{where} \quad \tilde w(r) = \begin{cases}
(\sgn(r))^{\frac{d+1}{2}} |r|^l & \text{ if } |x| \le 1, \\
w(r) & \text{ if } |r| \ge 1. 
\end{cases}. \]
Observe that if $w \in H^1_{sym}(\{ |r| \ge 1 \})$ in the sense that:
\[ \| w \|_{\dot H^1_{rad,l}(\{ |r| \ge 1 \})}^2 :=  \| \partial_r w \|_{L^2(\{ |r| \ge 1 \}, |r|^{d-1} dr)}^2 + l(l+d-2) \left\| \frac{w}{r} \right\|_{L^2(\{ |r| \ge 1 \}, |r|^{d-1} dr)}^2 < +\infty, \]
then $\tilde w \otimes Y_\ell = \q P (w \otimes Y_\ell) \in \dot H^1_{\ell}$, and the map $w \mapsto \tilde w$ is a continuous linear map.

\begin{lem} \label{N1_stab_Delta}
Consider the operator $\tilde \Delta: f \mapsto \q P \left( (\Delta f)|_{\{ |x| \ge 1\}} \right)$. 
Then $\tilde \Delta$ (is well defined and) maps $\q N_{1}^1 \cap \dot H^1_{\ell}$ to itself: for all $f = w \otimes Y_\ell \in\q N_{1}^1 \cap \dot H^1_{\ell}$, we also have $\tilde \Delta f \in \q N_{1}^1 \cap \dot H^1_{\ell}$.
\end{lem}

\begin{proof}
If suffices to prove it for $f = w \otimes Y_\ell$, where $w$ is of the form \eqref{def:w_ja_1}. A direct computation (like \eqref{formuLaplpowerlog}) yields that $(\Delta f)|_{|\{ |x| \ge 1 \}} = v \otimes Y_\ell$ where $v \in \q C(\{ |r| \ge 1 \})$, and recalling the form of the Laplacian in spherical coordinates, $v \in \dot H^1_{rad,l}(\{ |r| \ge 1 \})$. Now, from \eqref{actionDeltaRadonDhs}, we also have that
\[ \Td \Delta(w \otimes Y_\ell) = \partial_s^2 \Td(w \otimes Y_\ell), \]
so that $\partial_s \Td \Delta(w \otimes Y_\ell)  =0$ on $\{ |s| >1 \}$. By Proposition \ref{propdsTH1}, $\partial_s \Td (\q P (v \otimes Y_\ell))$ and $\partial_s \Td \Delta(w \otimes Y_\ell)$ coincide on $\{ |s| >1 \}$, and are both $0$ there. Hence $\tilde \Delta f = \q P (v \otimes Y_\ell) \in \q N_1^0$.
\end{proof}

\bigskip

\emph{Step 5: $\q N_1^1 \cap L^2_\ell$ is spanned by the $f_k$.}

\begin{lem} \label{lem:B_1}
Let $(p,\alpha) \in \tilde{\q B}$. Then $p=0$, $\alpha < -d/2+1$ and $\alpha = -l - d + 2(k+1)$ for some $k \in \m N$.
\end{lem}

\begin{proof}
If $w \in \q N_{1,\ell}^0=$ is such that
\[ 
\forall r \ge 1, \quad w(r) = \ln(r)^p r^\alpha,
\] 
the condition that $w \in \dot H^1_{rad,l}$ writes $\alpha < -d/2+1$. 

Then as $\tilde \Delta$ acts as $\Delta$ on $|r| \ge 1$ (like $\check \Delta$ does), the proof of Lemma \ref{lem:B_0} works word for word.
\end{proof}

Let $\tilde N_\ell \in \m N$ be the maximum of the $k$ such that $(0,-l-d+2k) \in \q B$. Then by applying repetitively $\tilde \Delta$ (and using \eqref{formuLaplpowerlog}) to $w \otimes Y_\ell$ where $w \in \dot H^1_{rad,l}$ and $w(r) = r^{N_\ell}$ for $r \ge 1$, we get that for all $k \in \llbracket 1, \tilde N_\ell \rrbracket$, $(0,-l-d+2k) \in \tilde{\q B}$.
Recalling the definition of the $f_k$ \eqref{def:fk}, we can reformulate this by saying that
\begin{align} \label{N11_span_1}
\q N_{1}^1 \cap \dot H^1_\ell = \Span (f_k; k \in \llbracket 1, \tilde N_\ell \rrbracket).
\end{align}

\bigskip

\emph{Step 6: Conclusion.}

\begin{lem}
Let $k \in \m N$ such that $\alpha_k < -d/2+1$. Then $f_k \in \q N_{1}^1 \cap \dot H^1_\ell$.
\end{lem}

\begin{proof}
Observe that, as $\alpha_k < -d/2+1$, $f_k \in \dot H^1_\ell$. Moreover, due to \eqref{formuLaplpowerlog} (with $p=0$), for $k \ge 1$,
\[ \m 1_{|x| >1} \Delta f_k = d_k \m 1_{|x|>1}f_{k-1} \quad \text{for some } d_k \in \m R, \]
and by definition of $\alpha_0$ and spherical harmonic,
\[ \m 1_{|x| >1} \Delta f_0 =0. \]

Therefore, due to Proposition \ref{propdsTH1}, there hold
\[ (\partial_s \Td (\m 1_{|x| >1} \Delta g_k))|_{\{ s > 1 \}} =  (\partial_s \Td ( \Delta g_k))|_{\{ s > 1 \}} = (\partial_s^2 (\partial_s \Td g_k))_{\{ s > 1 \}}. \]
As $g_k \in \dot H^1_\ell$, denote $h_k \in L^2(\{ |s| > 1 \})$ such that $(\partial_s \Td f_k)|_{\{|s| > 1\}} = h_k \otimes Y_\ell$: we obtained for $k \ge 1$ that
\begin{align} \label{diff:hk2}
\partial_s^2 h_k = d_k h_{k-1}.
\end{align}

Let us prove by induction on $k$ (such that $\alpha_k <-d/2+1$) that $h_k=0$.

For $k=0$, the function $u(t,x)= f_0(x)$ is solution of $\Box u=0$ on  $\{ |x| > |t|+1 \}$ (outside of the light cone). In particular, by finite speed of propagation, the solution $w$ to the wave equation with initial data $(f_0,0)$, satisfies $w(t,x)=f_{0}(x)$ on $\{ |x| > |t|+1 \}$. In particular, \eqref{equivfinallemmaH1} gives 
\begin{align*}
0=2 \lim_{t \to +\infty} \nor{\nabla w}{L^2(|x| \ge t+1)}^2  & =  \nor{\partial_s \Td f_0}{L^2(\{s \ge 1\}\times \m S^{d-1})}^2,
\end{align*}
This proves that $h_0=0$.

Let $k \ge 1$ and assume that $h_{k-1} =0$. In view of \eqref{diff:hk2}, we infer that there exists $\alpha, \beta \in \m R$ such that $h_k(s) = \alpha s+ \beta$ for $s >1$. As $h_k \in L^2(\{ |s| > 1 \},ds)$, $\alpha=\beta =0$ and $h_k =0$. This completes the induction.

Again, $h_k=0$ precisely means that $\m 1_{|s| \ge 1} \partial_s \Td f_k  = 0$, and thus $f_k \in \q N_1^1$.
\end{proof}

This concludes the proof of Theorem \ref{thmKer}.
\end{proof}

We can now complete the:

\begin{proof}[Proof of Theorem \ref{thmexteriorgene}]
The relations \eqref{proj:L2} and \eqref{proj:H1} write
\begin{align*}
\MoveEqLeft \| (u_0,u_1) \|_{\dot H^1 \times L^2}^2  = \| u_0 \|_{\dot H^1}^2 + \| u_1 \|_{L^2}^2 \\
& = \| \m 1_{|s| \ge R} \partial_s \Td u_0 \|_{L^2(\m R \times \m S^{d-1})}^2 + \| \pi_R^1 u_0 \|_{\dot H^1}^2 + \| \m 1_{|s| \ge R} \Td u_1 \|_{L^2(\m R \times \m S^{d-1})}^2 + \| \pi_R^0 u_1 \|_{L^2}^2.
\end{align*}
First, due to \eqref{eq:ortho_u0_u1} and symmetry, there hold
\begin{align*}
E_{ext,R}(u) & = \|  \partial_s \Td u_0 \|_{L^2([R,+\infty)\times \m S^{d-1})}^2 + \| \Td u_1 \|_{L^2([R,+\infty)\times \m S^{d-1})}^2 \\
& =\frac{1}{2} \left(\| \m 1_{|s| \ge R} \partial_s \Td u_0 \|_{L^2(\m R \times \m S^{d-1})}^2 + \| \m 1_{|s| \ge R} \Td u_1 \|_{L^2(\m R \times \m S^{d-1})}^2\right)
\end{align*}
Second, 
\[ \ker \pi_R = P(R) = K_R^1 \times K_R^0 = \ker \pi_R^1 \times \ker \pi_R^0, \]
so that
\[ \| \pi_R(u_0,u_1) \|_{\dot H^1 \times L^2}^2 = \| \pi_R^1 u_0 \|_{\dot H^1}^2 +  \| \pi_R^0 u_1 \|_{L^2}^2. \]
Therefore, we conclude that
\begin{align*}
\| (u_0,u_1) \|_{\dot H^1 \times L^2}^2  = 2E_{ext,R}(u) +  \| \pi_R(u_0,u_1) \|_{\dot H^1 \times L^2}^2.
\end{align*}
This is \eqref{channelprojectgene}. It remains to describe $u$ when $u_0 \in K_1^1$ and $u_1 \in K_1^0$, on the outer cone $\q C_1 := \{ |x| \ge t+1 \}$ for $t \ge 0$ (the case $t \le 0$ being treated with data $(u_0,-u_1)$ and by scaling, we get the description for any $R>0$). For this, it suffices to compute the solution $v_k$ to \eqref{eq:lw} with initial data $(f_k,0)$ on $\q C_1$ for any $k \in \m N$ (notice that for large $k$, these solutions to the wave equation do not belong to $\dot H^1 \times L^2$). Indeed, for $k \ge 0$, then $\partial_t v_{k+1}$ is the solution to \eqref{eq:lw} with initial data $(0, \Delta f_{k+1})$. As $(\Delta f_{k+1})|_{\{ |x| >1 \}} = c_k g_{k}|_{\{ |x| >1 \}}$ with $c_k \ne 0$ for $k \ge 1$, so that by finite speed of propagation $\frac{1}{c_{k}} \partial_t v_{k+1}$ coincide on on $\q C_1$ with the solution to \eqref{eq:lw} with initial data $(0,g_k)$.

We prove by induction on $k \in \m N$ that there exists $\alpha_{k,j} \in \m R$ for $j \in \llbracket 0,k \rrbracket$ such that
\begin{align} \label{ind:vk}
\forall (t,x) \in \q C_1, \quad v_k(t,x) = \sum_{j=0}^{k} \alpha_{k,j} t^{2(k-j)} f_j(x).
\end{align}
For $k=0$, simply recall that $(\Delta f_0)|_{|x|>1} =0$ so that for $(t,x) \in \q C_1$, $v_0(t,x) = f_0(x)$.

Assume that \eqref{ind:vk} holds for some $k \ge 0$, and let us prove it for $k+1$. Observe that $\Delta v_{k+1}$ is a solution to \eqref{eq:lw}, with initial data $(\Delta f_{k+1},0)$. As $(\Delta f_{k+1})|_{\{|x|>1\}} = c_k f_k|_{\{ |x| > 1 \}}$ for some $c_k \in \m R$, by uniqueness in the Cauchy problem for \eqref{eq:lw} and finite speed of propagation, we infer that
\[ \forall  (t,x) \in \q C_1, \quad \partial_{tt} v_{k+1}(t,x) = \Delta v_{k+1}(t,x) = c_k v_k(t,x). \]
By induction hypothesis, we infer that
\[ \forall  (t,x) \in \q C_1, \quad \partial_{tt} v_{k+1}(t,x) = \sum_{j=0}^{k} \alpha_{k,j} t^{2(k-j)} f_j(x). \]
Integrating in time twice for each fixed $x$, (with $\partial_t v_{k+1}(0,x) = 0$ and $v_{k+1}(0,x) = f_{k+1}(x)$, we get
\[ \forall  (t,x) \in \q C_1, \quad v_{k+1}(t,x) = c_k \sum_{j=0}^{k} \frac{\alpha_{k,j}}{(2(k-j)+1)(2(k-j)+2)} t^{2(k-j)+2} f_j(x) + f_{k+1}(x). \]
 This ends the induction step. Notice that $\alpha_{k,k}=1$ (by induction or by evaluation at $t=0$).

The proof of Theorem \ref{thmexteriorgene} is complete.
\end{proof}

\appendix

\section{Computations of the laplacian of some functions} \label{app:a}

 Using the Laplacian in polar coordinates
 \begin{align}  
 \label{polarLap}
 \Delta f=\frac{1}{r^{d-1}}\frac{\partial}{\partial r}\left(r^{d-1}\frac{\partial f}{\partial r}\right)+\frac{1}{r^{2}}\Delta_{\m S^{d-1}}f=\frac{\partial^{2} f}{\partial r^{2}}+\frac{ d-1}{r}\frac{\partial f}{\partial r}+\frac{1}{r^{2}}\Delta_{\m S^{d-1}}f
 \end{align} 
We first compute
\begin{align*}
\frac{1}{r^{d-1}}\frac{\partial}{\partial r}\left(r^{d-1}\frac{\partial }{\partial r}r^{\alpha}\right)=\alpha (\alpha+d-2)r^{\alpha-2}.
\end{align*}
In particular, since $|x|^{l}Y_{\ell}$ is a harmonic polynomial, we have
\[ 0=\Delta \left[r^{l}Y_{\ell}\right]=l (l+d-2)r^{l-2}Y_{\ell}+\frac{r^{l}}{r^{2}}\Delta_{\m S^{d-1}}Y_{\ell}, \]
 which gives $\Delta_{\m S^{d-1}}Y_{\ell}=-\lambda_{l}Y_{\ell}$, with $\lambda_{l}=l (l+d-2)$. It gives also 
  \begin{align}  
\label{formuLaplpower}
\Delta \left[r^{\alpha}Y_{\ell}\right]=\left[\alpha (\alpha+d-2)-l (l+d-2)\right]r^{\alpha-2}Y_{\ell}
 \end{align} 
 We also compute
 \begin{multline*}  
 \frac{1}{r^{d-1}}\frac{\partial}{\partial r}\left(r^{d-1}\frac{\partial }{\partial r}\log(r)^{p}r^{\alpha}\right)\\
 = r^{\alpha-2}\left[ \alpha (\alpha+d-2)\log(r)^{p} +p(2\alpha+d-2) \log(r)^{p-1}+p(p-1)\log(r)^{p-2} \right]
 \end{multline*} 
 which gives  \eqref{formuLaplpowerlog}, namely
 \begin{multline*}  
\Delta \left[\log(r)^{p}r^{\alpha}Y_{\ell}\right] = r^{\alpha-2}Y_{\ell}\Big[\left[\alpha (\alpha+d-2)-l (l+d-2)\right]\log(r)^{p} \\
\left. +p(2\alpha+d-2) \log(r)^{p-1}+p(p-1)\log(r)^{p-2} \right].
\end{multline*}

\bibliographystyle{alpha} 
\bibliography{biblio}

\end{document}